\documentclass[11pt,reqno]{amsart}

\usepackage{amsmath, amsthm, amssymb}
\usepackage{amsfonts}
\usepackage{bm,mathrsfs}
\usepackage{pifont}
\usepackage{enumitem}
\usepackage{graphicx}
\usepackage{subcaption}

\usepackage{hyperref}

\usepackage[ansinew]{inputenc}
\usepackage[dvips]{epsfig}
\usepackage{graphicx}
\usepackage[english]{babel}
\usepackage{mathrsfs}

\usepackage{thmtools}
\theoremstyle{plain}
\declaretheorem[title=Theorem, parent=section]{theorem}
\declaretheorem[title=Lemma,sibling=theorem]{lemma}
\declaretheorem[title=Proposition,sibling=theorem]{proposition}

\declaretheorem[title=Question, parent=section]{question}
\declaretheorem[title=Remark, parent=section]{remark}
\declaretheorem[title=Definition,parent=section]{definition}

\theoremstyle{definition}
\declaretheorem[title=Example, sibling=theorem]{example}
\declaretheorem[title=Assumption, numbered=no]{assumption*}

\numberwithin{equation}{section}

\usepackage[backgroundcolor=white, bordercolor=blue,
linecolor=blue]{todonotes}

\usepackage{dsfont}
\usepackage{bbm}

\newcommand{\average}{{\mathchoice {\kern1ex\vcenter{\hrule height.4pt
width 6pt depth0pt} \kern-9.7pt} {\kern1ex\vcenter{\hrule
height.4pt width 4.3pt depth0pt} \kern-7pt} {} {} }}

\begin{document}
\allowdisplaybreaks
\title{Global $ C^{1}$ regularity for degenerate fully nonlinear elliptic equations}

\author{Jiangwen Wang}
\author{Feida Jiang$^*$}

\address{School of Mathematics and Shing-Tung Yau Center of Southeast University, Southeast University, Nanjing 211189, P.R. China}
\email{\url{jiangwen\_wang@seu.edu.cn}}

\address{School of Mathematics and Shing-Tung Yau Center of Southeast University, Southeast University, Nanjing 211189, P.R. China; Shanghai Institute for Mathematics and Interdisciplinary Sciences, Shanghai, 200433, P.R. China}
\email{\url{jiangfeida@seu.edu.cn}}

\date{\today}
	\thanks{*corresponding author}



\allowdisplaybreaks

\begin{abstract}
In this paper, we prove global $C^{1}$ regularity for Dirichlet problem of a class of degenerate fully nonlinear elliptic equations on $ C^{2} $ domain. This corresponds to the boundary counterpart of the interior $ C^{1}$ regularity results by \cite{APPT22} and \cite{AN25}. By an example, we show that $ C^{1,\alpha} $ regularity of boundary datum is sharp within the scale of H\"{o}lder spaces. Furthermore, we also provide boundary $ C^{1,\beta} $ regularity for a class of singular fully nonlinear elliptic equations.
\end{abstract}

\allowdisplaybreaks

\maketitle

\section{Introduction}\label{Intro}
This paper deals with the following boundary value problem:
\begin{equation}
\label{Maineq1}
\left\{
     \begin{alignedat}{2}
       \big[\sigma_{1}(|Du|)+a(x)\sigma_{2}(|Du|)\big] F(D^{2}u) & = f        \quad  &&  \text{in}  \ \ \Omega     ,    \\
          u  &= g     \quad  &&  \text{on}   \ \  \partial \Omega ,        \\
     \end{alignedat}
     \right.
\end{equation}
where $ F: \mathrm{Sym}(n) \rightarrow \mathbb{R} $ denotes a uniformly $ (\lambda, \Lambda)$-elliptic operator, $ \sigma_{i}: [0,+\infty) \rightarrow [0,+\infty), i=1,2 $ stand for moduli of continuity, $ 0\leq  a(x) \in C^{0}(\Omega) $, $ f \in C^{0}(\Omega) \cap L^{\infty}(\Omega) $, and $ g \in C^{1,\alpha}(\partial \Omega) $ for some $ \alpha \in (0,1) $. In addition, $ \Omega \subset \mathbb{R}^{n} $ is assumed to be a bounded $ C^{2}$-domain.

Double-phase degenerate fully nonlinear equations arise naturally in numerous fields, including optimization and stochastic control. Models featuring nonhomogeneous degeneracy laws are particularly important in materials science, as they provide a more faithful characterization of intricate physical phenomena. These nonhomogeneous structures are crucial for capturing the diverse behavioral characteristics observed in practical applications. As a concrete example, equation \eqref{Maineq1} originates from image restoration problems \cite{HH21}.

A key motivation for investigating this class of degenerate problems stems from the calculus of variations, particularly the model proposed by Zhikov in \cite{Z93} for the study of homogenization theory and the Lavrentiev phenomenon:
\begin{equation*}
  u \longmapsto \int (|Du|^{p} + a(x)|Du|^{q})dx.
\end{equation*}
This functional serves as a reliable model for strongly anisotropic materials, whose gradient growth exponents govern position-dependent hardening behaviors.

The nonvariational counterparts of the models discussed above take the form of degenerate fully nonlinear equations. One of the most widely studied prototypes given by
\begin{equation}
\label{Se1:eq2}
|Du|^{p} F(D^{2}u) = f   \quad  \text{in}  \quad \Omega.
\end{equation}
The groundbreaking work was made by Imbert-Silvestre in \cite{IS13} where they showed an interior $ C^{1,\alpha} $ regularity of viscosity solution to \eqref{Se1:eq2}. Shortly after, in \cite{ART15}, Ara\'{u}jo {\em et al.} proved an optimal interior $ C^{1,\alpha} $ regularity when $ F $ is convex/concave. In a recent paper \cite{APPT22}, Andrade {\em et al.} considered a degenerate fully nonlinear equation of the form
\begin{equation}\label{Se1:eq3}
  \sigma(|Du|) F(D^{2}u) =f   \quad   \text{in}   \quad   \Omega,
\end{equation}
and established interior $ C^{1} $ regularity for solutions of \eqref{Se1:eq3} provided that $ \sigma: [0,+\infty) \rightarrow [0,+\infty)$ has an inverse whose modulus of continuity is Dini-continuous near the origin. Following the same idea
with more careful analysis, Andrade-Nascimento in \cite{AN25} extended the result of \cite{APPT22} to
\begin{equation}
\label{Se1:eq4}
  \big[\sigma_{1}(|Du|)+a(x)\sigma_{2}(|Du|)\big] F(D^{2}u) = f        \quad   \text{in}  \quad  \Omega.
\end{equation}

Much progress has been made in understanding the boundary regularity of solution to \eqref{Se1:eq2} under the Dirichlet boundary condition. For gradient H\"{o}lder boundary data on a $ C^{2} $ domain, Birindelli-Demengel \cite{BD14} established global $ C^{1,\alpha} $ regularity. Ara\'{u}jo-Sirakov \cite{AS23} further obtained sharp boundary $ C^{1,\alpha} $ regularity. In addition, the work in \cite{LL23} proves $C^{1,\alpha}$ regularity on $C^{1,\alpha}$ domains without employing the boundary flattening argument. For a thorough survey of boundary regularity results concerning Dirichlet problems, we refer interested readers to \cite{BBLL24, BSRR23}.

Thus a natural question arises:

\begin{question}\label{Se1:question1}
Can we establish the global $ C^{1} $ regularity for viscosity solutions of \eqref{Se1:eq4} under optimal boundary datum?
\end{question}

The main purpose of this paper is to provide a clear affirmative answer to question \ref{Se1:question1}. The obtained result is new even for the case $ a(x) \equiv 0 $.

Before proceeding further, we state the main assumptions of this paper.

\label{A1} {\bf (A1).} The operator $ F $ is uniformly elliptic with $F(\mathrm{O}_n) = 0$; namely, there exist constants $0 < \lambda \leq \Lambda < \infty $ such that
$$  \mathscr{P}^{-}_{\lambda,\Lambda}(\mathrm{N}) \leq F(\mathrm{M}+\mathrm{N}) - F(\mathrm{M})
  \leq \mathscr{P}^{+}_{\lambda,\Lambda}(\mathrm{N}),$$
for all $\mathrm{M}, \mathrm{N} \in \mathrm{Sym}(n)$,
where $\mathscr{P}^{\pm}$ denote the Pucci extremal operators in \cite{CC95}, namely,
\[
  \mathscr{P}^{+}_{\lambda,\Lambda}(\mathrm{X})
  := \lambda \sum_{e_{i}<0} e_{i}(\mathrm{X}) + \Lambda \sum_{e_{i}>0} e_{i}(\mathrm{X}),
  \quad \text{and} \quad
  \mathscr{P}^{-}_{\lambda,\Lambda}(\mathrm{X})
  := \lambda \sum_{e_{i}>0} e_{i}(\mathrm{X}) + \Lambda \sum_{e_{i}<0} e_{i}(\mathrm{X});
\]

\label{A2} {\bf (A2).} $ a(x) \geq 0 $ and $ a(x) \in C^{0}(\Omega) $;

\label{A3} {\bf (A3).} $ f \in C^{0}(\Omega) \cap L^{\infty}(\Omega) $ and $ g \in C^{1,\alpha}(\partial \Omega) $ for some $ \alpha \in (0,1) $;

\label{A4} {\bf (A4).} $ \Omega \subset \mathbb{R}^{n} $ is a bounded $ C^{2}$-domain;

\label{A5} {\bf (A5).} (Monotonicity of $ \sigma_{i} $, $ i=1,2$ )

--\label{A5a} {\bf (A5a)} We suppose that $ \sigma_{1}, \sigma_{2}: [0,+\infty) \rightarrow [0,+\infty) $ are continuous and monotone increasing with
\begin{equation*}
  \lim_{t \rightarrow 0} \sigma_{i}(t) = 0, \quad i=1,2;
\end{equation*}

--\label{A5b} {\bf (A5b)} We suppose that $ \sigma_{1}, \sigma_{2}: [0,+\infty) \rightarrow [0,+\infty) $ are continuous and monotone decreasing with
\begin{equation*}
  \lim_{t \rightarrow 0} \sigma_{i}(t) >0, \quad  i=1,2.
\end{equation*}

Furthermore, in both cases, we assume
\begin{equation*}
  \sigma_{1}(t) \geq \sigma_{2}(t), \quad t \in [0,1],
\end{equation*}
and, in particular,
\begin{equation*}
  \sigma_{1}(1)  \geq  \sigma_{2}(1)  \geq  1;
\end{equation*}

\label{A6} {\bf (A6).} The function $ \sigma_{2} $ admits an inverse $ \sigma_{2}^{-1} $ that is Dini continuous, namely,
\begin{equation*}
  \int_{0}^{\tau} \frac{\sigma_{2}^{-1}(t)}{t} dt < \infty,
\end{equation*}
for some $ \tau > 0 $.

Note that under \hyperref[A5a]{\bf (A5a)}, equation \eqref{Se1:eq4} is degenerate elliptic wherever $Du=0$, whereas under \hyperref[A5b]{\bf (A5b)} it may be singular there.

The assumption \hyperref[A4]{\bf (A4)} was inspired by the approach introduced in \cite{BD14, BBLL24}. Specifically, we may assume that $ 0 \in \partial \Omega $, and there exists a ball $ B_{1} \subset \mathbb{R}^{n} $ and $ \varphi \in C^{2}(\mathbb{R}^{n-1}) $ such that $ \varphi(0) = 0$, $  \nabla \varphi(0) = 0 $ and
\begin{equation*}
  \Omega \cap B_{1}  \subset \{ y \in B_{1}: y_{n} > \varphi(y')   \},  \quad    \partial \Omega \cap B_{1} =  \{ y \in B_{1}: y_{n} = \varphi(y')   \}.
\end{equation*}

It is well known that $C^{2} $-domains automatically satisfy a uniform interior sphere condition, see \cite[Lemma 2.2]{AKSZ07} for details.

\subsection{Statement of the main results}
We now establish the first main result, concerning the global $ C^{1}$ regularity of viscosity solution to \eqref{Maineq1}.

\begin{theorem}[{Global borderline regularity: degenerate case}]
\label{Thm1}
Let $ u \in C^{0}(\overline{\Omega}) $ be a viscosity solution to \eqref{Maineq1}. Assume that \hyperref[A1]{\bf (A1)}--\hyperref[A5a]{\bf (A5a)} and \hyperref[A6]{\bf (A6)} hold. Then $ u \in C^{1}(\overline{\Omega}) $ and there exists a modulus of continuity $ \omega:  [0,+\infty) \rightarrow [0,+\infty)$ depending only on dimension $ n $, $ \lambda, \Lambda $, $ ||u||_{L^{\infty}(\Omega)}, ||f||_{L^{\infty}(\Omega)}, ||g||_{C^{1,\alpha}(\partial \Omega)} $, $ \sigma_{1} $ and $ \sigma_{2} $, such that
\begin{equation*}
  |Du(x) - Du(y)| \leq \omega(|x-y|)
\end{equation*}
for every $ x,y \in \overline{\Omega} $.
\end{theorem}

Next we state the second main result of this paper.

\begin{theorem}
[{Global $ C^{1,\beta}$ regularity: singular case}]
\label{Thm2}
Let $ u \in C^{0}(\overline{\Omega}) $ be a viscosity solution to \eqref{Maineq1}. Assume that \hyperref[A1]{\bf (A1)}--\hyperref[A4]{\bf (A4)} and \hyperref[A5b]{\bf (A5b)} hold. Then $ u \in C^{1,\beta}(\overline{\Omega}) $ for some $ \beta \in (0,1) $.
\end{theorem}

\begin{remark}
We make some remarks on the above theorems.

\begin{enumerate}

\item In Theorem \ref{Thm1}, with the stronger assumption that $ \sigma_{2}^{-1} $ behaves as a H\"{o}lder continuous function near the origin, the corresponding solutions admit global $ C^{1,\alpha_{1}}$ regularity, for some $ \alpha_{1} \in (0,1) $. To a certain extent, Theorem \ref{Thm1}  generalizes the results established in \cite{BD14}, and can be regarded as the boundary counterpart of the interior $ C^{1,\alpha'}$ regularity obtained in \cite{IS13} and \cite{De21}.

\item Theorem \ref{Thm1} can also be extended to multi-phase degenerate fully nonlinear equations, namely,
  \begin{equation*}
\left\{
     \begin{alignedat}{2}
        \bigg[ \sigma_{1}(|Du|) + \sum_{i=2}^{k} a_{i}(x)\sigma_{i}(|Du|)                   \bigg] F(D^{2}u)&= f  \quad  &&  \text{in}  \ \  \Omega     ,    \\
          u & = g     \quad  &&  \text{on}   \ \  \partial \Omega.         \\
     \end{alignedat}
     \right.
\end{equation*}
As long as we consider the following decay:
\begin{equation*}
  \sigma_{1}(t) \geq \sigma_{2}(t) \geq \cdots \geq \sigma_{k}(t), \ \forall \  t\in [0,1], \ \      \sigma_{1}(1) \geq \sigma_{2}(1) \geq \cdots \geq \sigma_{k}(1) \geq 1,
\end{equation*}
\begin{equation*}
  \lim_{t\rightarrow 0} \sigma_{i}(t) = 0, \quad  i= 1,2,\cdots, k, \quad  0 \leq a_{j}(x) \in C^{0}(B_{1}), \  j=2,\cdots, k,
\end{equation*}
and $ \sigma_{k}^{-1} $ satisfies Dini condition, and $ g \in C^{1,\alpha}(\partial \Omega) $ for some $ \alpha \in (0,1) $. Notably, the monotonicity of $ \{\sigma_{k}\}_{k}$ readily implies that $ \{\sigma_{j}^{-1}\}_{j=1}^{k-1} $ is also Dini continuous.

\vspace{2mm}

\item Theorems \ref{Thm1} and \ref{Thm2} also admit an extension to the setting of degenerate or singular fully nonlinear equations with Hamiltonian terms
\begin{equation*}
\left\{
     \begin{alignedat}{2}
       \big[\sigma_{1}(|Du|)+a(x)\sigma_{2}(|Du|)\big] F(D^{2}u) + H(x,Du) & = f        \quad  &&  \text{in}  \ \ \Omega     ,    \\
          u &  = g     \quad  &&  \text{on}   \ \  \partial \Omega ,        \\
     \end{alignedat}
     \right.
\end{equation*}
where $ H: \Omega \times \mathbb{R}^{n} \rightarrow \mathbb{R} $ is continuous and there is a constant $ \mathcal{M} $ such that
\begin{equation*}
  |H(x,\xi)| \leq \mathcal{M}(1+\sigma_{2}(|\xi|))
\end{equation*}
for every $ x \in \Omega $ and $ \xi \in \mathbb{R}^{n} $.

\vspace{2mm}

\item  As discussed above, the assumption \hyperref[A5b]{\bf (A5b)} in Theorem \ref{Thm2} encompasses singular fully nonlinear equations. Illustrative examples may be obtained by choosing, for example,
\begin{equation*}
  \sigma_{1}(t) = t^{p}, \quad  \sigma_{2}(t) = t^{q}, \quad  -1 < p \leq q <0,
\end{equation*}
or
\begin{equation*}
  \sigma_{1}(t)= \frac{\ln^{\beta}(1+t)+1}{t}, \ \ \sigma_{2}(t)= \frac{\ln^{\beta}(1+t)}{t},\ \ 0 < \beta \leq 1 .
\end{equation*}
For more examples, we refer the reader to \cite{WJ26, BBO23}.
\end{enumerate}
\end{remark}

\subsection{The ideas of proof of Theorems \ref{Thm1} and \ref{Thm2}}
We first note that establishing the boundary $ C^{1} $ regularity of the viscosity solution $ u$ to \eqref{Maineq1} is equivalent to showing that the graph of $ u $ can be approximated by an affine function with an error bounded by $ Cr $ on every $ {B_{r}\cap \{y_{n}>\varphi(y')\}} $. To this end, we first characterize the boundary behavior of solutions to \eqref{Maineq1} via the $ C^{2}$-distance function. Combining with the Ishii-Lions method, this argument yields the boundary $ C^{0,\gamma} $ and $ C^{0,1} $ regularity of viscosity solutions to the perturbed equation
\begin{equation}
\label{intro:eq15}
\left\{
     \begin{alignedat}{2}
       \big[\sigma_{1}(|Du+\xi|)+a(y)\sigma_{2}(|Du+\xi|)\big] F(D^{2}u) & = f        \quad  &&  \text{in} \ \ B_{1} \cap \{y_{n} > \varphi(y')\}    ,    \\
          u(y)  & = g(y)    \quad  &&   \text{on}   \ \  B_{1}  \cap \{y_{n} = \varphi(y')\},        \\
     \end{alignedat}
     \right.
\end{equation}
where $ \xi \in \mathbb{R}^{n} $ is an arbitrary vector. We then derive a boundary approximation lemma by comparing solutions of \eqref{Maineq1} and \eqref{intro:eq15} with those of the homogeneous uniformly elliptic problem $ F(D^{2}u) = 0 $ defined on $ {B_{r}\cap \{y_{n}>\varphi(y')\}} $, for all $ r \in (0,1) $. The main technical difficulty lies in the fact that the degenerate coefficient $ \sigma_{1} + a(\cdot)\sigma_{2} $  fails to satisfy the {\it non-collapsing} property. To overcome this obstacle, we adopt the strategy from \cite{AN25} and construct a new sequence of reinforced continuous moduli $ \{\sigma_{1}^{j}(t)+a_{j}(\cdot)\sigma_{2}^{j}(t)\}_{j \in \mathbb{N}} $. Finally, a refined iteration and convergence analysis yeilds the desired boundary $ C^{1} $ regularity.

To prove Theorem \ref{Thm2}, we adopt the Ishii-Lions method to establish the boundary $C^{0,1}$ regularity of solutions to the singular problem \eqref{intro:eq15} with $\xi = 0$. This reduction simplifies the original problem to a classical fully nonlinear Dirichlet problem. We then apply \cite[Theorem 2.1]{AS23} to obtain the boundary $C^{1,\beta}$ regularity of solutions to \eqref{Maineq1}.

\vspace{2mm}

This paper is organised as follows. In Section \ref{Section ex1}, we present an example, which shows that the boundary condition $ g \in C^{1,\alpha} $ ($ \alpha > 0 $) is needed for the boundary $ C^{1} $ regularity. In Section \ref{Section 2}, we present the definition of viscosity solution to \eqref{Maineq1} and several auxiliary lemmas. In Section \ref{Section 3}, we deliver the boundary Lipschitz regularity of viscosity solution to perturbation equation \eqref{Section3:eq1} by using scaling technique, barrier function, comparison principle and Ishii-Lions method, see Theorem \ref{Se4:thm4}. Finally, Section \ref{Section 4} provides the proofs of Theorems \ref{Thm1} and \ref{Thm2}.

\section{A counterexample}\label{Section ex1}
In this section, we show an example to show that $C^{1}$ boundary data alone do not ensure global $C^{1}$ regularity of the solution, even for a particularly simple equation in the class considered in this paper. In the example, we shall take $a\equiv 0$ and $F(M)=\operatorname{tr}(M)$ in \eqref{Maineq1}.

\begin{example}[$C^{1}$ boundary data with an unbounded boundary gradient]
\label{Se2:counterexample}
Let $\Omega=B_1(0)\subset\mathbb{R}^{2}$ and set
\begin{equation}\label{Se2:sigma}
  \sigma(t):=\frac{2t}{1+t}, \qquad t\geq 0.
\end{equation}
For $(x,y)=(r\cos\theta,r\sin\theta)$, define
\begin{equation}\label{Se2:u-definition}
  u(r,\theta):=u_h(r,\theta)+u_p(r),
\end{equation}
where
\begin{equation}\label{Se2:uh-up}
  u_h(r,\theta)
  :=\sum_{n=2}^{\infty}\frac{r^n}{n^2\ln n}\cos(n\theta),
  \qquad
  u_p(r):=\frac14(1-r^2).
\end{equation}
On $\partial B_1(0)$, prescribe
\begin{equation}\label{Se2:boundary-datum}
  g_h(\cos\theta,\sin\theta)
  :=\sum_{n=2}^{\infty}\frac{\cos(n\theta)}{n^2\ln n}.
\end{equation}
Then $g_h\in C^1(\partial B_1(0))$, and $u\in C^0(\overline{B_1(0)})\cap C^\infty(B_1(0))$ is a classical, and hence viscosity, solution of the Dirichlet problem
\begin{equation}\label{Se2:counterexample-problem}
  \left\{
  \begin{aligned}
    \sigma(|Du|)\Delta u
      &=f(x,y):=-\frac{2|Du(x,y)|}{1+|Du(x,y)|}
        &&\text{in }B_1(0),\\
    u&=g_h &&\text{on }\partial B_1(0).
  \end{aligned}
  \right.
\end{equation}
Moreover, $f\in C^0(B_1(0))\cap L^\infty(B_1(0))$, while
\begin{equation}\label{Se2:gradient-blowup}
  \lim_{r\rightarrow 1^-}\partial_r u(r,0)=+\infty.
\end{equation}
Therefore, $u\notin C^1(\overline{B_1(0)})$.
\end{example}

\begin{proof}
We verify the assertions in the order in which they appear in the example. First, $\sigma$ is continuous and increasing, $\sigma(0)=0$, and $\sigma(1)=1$. Its inverse on $[0,2)$ is $\sigma^{-1}(s)=\frac{s}{2-s}$. Consequently,
\[
  \int_0^1\frac{\sigma^{-1}(s)}{s}\,ds
  =\int_0^1\frac{1}{2-s}\,ds
  =\ln 2<\infty,
\]
so $\sigma^{-1}$ satisfies the required Dini condition near the origin.

For every $n\geq 2$, the function $r^n\cos(n\theta)$ is harmonic in $B_1(0)$. Since $\sum_{n=2}^{\infty}(n^2\ln n)^{-1}<\infty$, the series defining $u_h$ converges absolutely and uniformly on $\overline{B_1(0)}$; in particular, $u_h\in C^0(\overline{B_1(0)})$. Its convergence is locally uniform in $B_1(0)$, so $u_h$ is harmonic there. Also,
\[
  \Delta u_p=-1 \quad\text{in }B_1(0),
  \qquad
  u_p=0 \quad\text{on }\partial B_1(0).
\]
It follows that $\Delta u=-1$ in $B_1(0)$ and that $u=g_h$ on $\partial B_1(0)$. \eqref{Se2:counterexample-problem} now follows immediately from \eqref{Se2:sigma} together with the above identities. Since $u$ is smooth in $B_1(0)$, the function $f=-\sigma(|Du|)$ is continuous there; moreover, $|f|\leq 2$. Then, $f\in C^0(B_1(0))\cap L^\infty(B_1(0))$.

It remains to check the regularity of $g_h$ and the failure of boundary gradient control. The series in \eqref{Se2:boundary-datum} converges absolutely and uniformly. Formally differentiating it gives
\begin{equation}\label{Se2:gh-derivative}
  \frac{d}{d\theta}g_h(\cos\theta,\sin\theta)
  =-\sum_{n=2}^{\infty}\frac{\sin(n\theta)}{n\ln n}.
\end{equation}
The coefficients $a_n=(n\ln n)^{-1}$ decrease to zero and satisfy $n a_n\to 0$. The classical uniform-convergence criterion for sine
series therefore shows that the series on the right-hand side of \eqref{Se2:gh-derivative} converges uniformly; see
\cite[Vol.~I, Chapter~V, Theorem~1.3]{Z02}. Termwise differentiation is thus justified, and $g_h\in C^1(\partial B_1(0))$.

Finally, along the radius $\theta=0$,
\[
  \partial_r u(r,0)
  =\sum_{n=2}^{\infty}\frac{r^{n-1}}{n\ln n}-\frac r2.
\]
All terms in the series are nonnegative. Hence, by monotone convergence,
\[
  \lim_{r\rightarrow 1^-}
  \sum_{n=2}^{\infty}\frac{r^{n-1}}{n\ln n}
  =\sum_{n=2}^{\infty}\frac{1}{n\ln n}=+\infty,
\]
which proves \eqref{Se2:gradient-blowup}; see Figure \ref{fig:double}.
\end{proof}

\begin{figure}[htbp]
  \centering
  \begin{subfigure}[b]{0.48\textwidth}
    \centering
    \includegraphics[width=\linewidth]{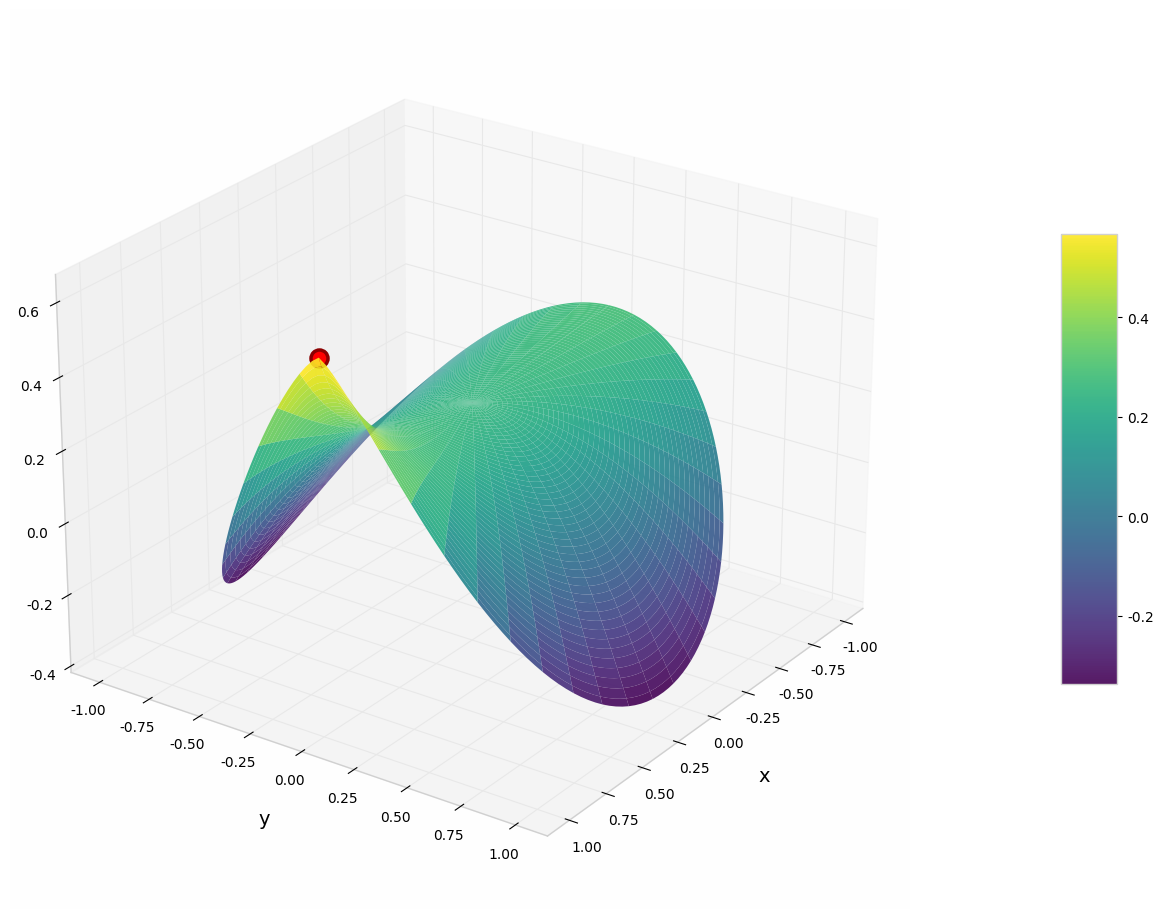}
    \caption{The graph of $u$ over $B_1$. The red point indicates the boundary point $(1,0)$ approached along $\theta=0$.}
    \label{image1}
  \end{subfigure}
  \hfill
  \begin{subfigure}[b]{0.48\textwidth}
    \centering
    \includegraphics[width=\linewidth]{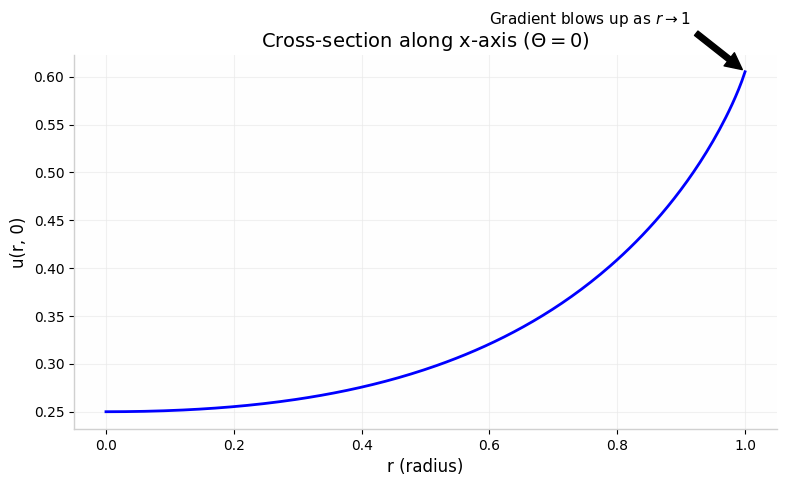}
    \caption{The radial profile $r\mapsto u(r,0)$. Its slope becomes unbounded as $r\rightarrow 1^-$.}
    \label{image2}
  \end{subfigure}
  \caption{The solution remains continuous up to the boundary, whereas its radial derivative blows up at $(1,0)$.}
  \label{fig:double}
\end{figure}

\begin{remark}
The harmonic part $u_h$ of the construction in Example \ref{Se2:counterexample} is adapted from a classical example in the theory of conjugate trigonometric series; see Chaundy and Jolliffe \cite{CJ17}, Zygmund \cite{Z02}, and Gissy-Miihkinen-Virtanen \cite{GMV21}. We add the particular solution \(u_p(r)=\frac14(1-r^2)\) and choose the degeneracy law \(\sigma(t)=2t/(1+t)\) in order to embed this classical harmonic example into the class of degenerate equations \eqref{Se2:counterexample-problem} considered here.
\end{remark}

\begin{remark}[Interpretation of the example]
The construction in Example \ref{Se2:counterexample} proves that the assumption $g\in C^{1,\alpha}(\partial\Omega)$ in Theorem~\ref{Thm1} cannot in
general be weakened to $g\in C^1(\partial\Omega)$. Thus, within the H\"older scale, the requirement of a positive exponent is necessary. The example does not settle the intermediate case $g\in C^{1,\mathrm{Dini}}(\partial\Omega)$; whether such data suffice for global $C^1$ regularity remains an interesting question.

From the construction of Example \ref{Se2:counterexample}, we know that the function $u(r, \theta)$ defined in \eqref{Se2:u-definition}--\eqref{Se2:boundary-datum} is also a solution of the problem
\begin{equation}\label{Se2:counterexample-problem-2}
  \left\{
  \begin{aligned}
    \Delta u
      &=-1
        &&\text{in }B_1(0),\\
    u&=g_h\in C^{1}(\partial B_1(0)) &&\text{on }\partial B_1(0).
  \end{aligned}
  \right.
\end{equation}
This implies that, even for nondegenerate uniformly elliptic equations, merely $C^1$ boundary data can not ensure global $C^1$ regularity of the solution.
\end{remark}

\section{Preliminaries}
\label{Section 2}
In this section, we first present the definitions of viscosity solution, non-collapsing set and shored-up sequence. Then we give some useful auxiliary lemmas in the proof of Theorem \ref{Thm1}.

\subsection{Some definitions}

\begin{definition}\label{subsection 2.1}
A function $ u \in C^{0}(\Omega) $ is a viscosity super-solution (resp. sub-solution) to \eqref{Maineq1} if whenever $ \varphi \in C^{2}(\Omega) $ and $ x_{0} \in \Omega $ such that $ u-\varphi $ has a local minimum (resp. a local maximum) at $ x_{0} $, then
\begin{equation*}
  \big[\sigma_{1}(|D\varphi(x_{0})|)+a(x_{0})\sigma_{2}(|D\varphi(x_{0})|)\big] F(D^{2}\varphi(x_{0})) \leq  f(x_{0})  \quad (\text{resp}. \cdots \geq f(x_{0})).
\end{equation*}
Finally, we say that $ u $ is a viscosity solution to \eqref{Maineq1} if it is both a viscosity sub-solution and a viscosity super-solution.
\end{definition}

\begin{definition}
Let $ \boldsymbol{\mathscr{Q}} $ be a collection of moduli of continuity defined over an interval $ I \in \mathcal{I} $, where $ \mathcal{I}:= \{(0,T]: 0< T < \infty \} \cup \{\mathbb{R}_{0}^{+}\}  $. We call that $ \boldsymbol{\mathscr{Q}} $ is a non-collapsing set if for all sequences $ (f_{k})_{k \in \mathbb{N}} \subset \boldsymbol{\mathscr{Q}} $, and all sequences $ (f_{k})_{k \in \mathbb{N}} \subset I $, then
\begin{equation*}
 \lim_{k\rightarrow +\infty} f_{k}(a_{k}) = 0 \ \  \Rightarrow \ \
 \lim_{k\rightarrow +\infty} a_{k} = 0.
\end{equation*}
\end{definition}

To construct a family of non-collapsing moduli of continuity $\boldsymbol{\mathscr{Q}} $ in Section \ref{Subsection4.1}, and   similarly to \cite[Section 7]{APPT22} and \cite[Section 5.1]{AN25}, we also introduce the notation of {\it shored-up}.

\begin{definition}
A sequence of moduli of continuity $(\sigma_{1}^{k})_{k \in \mathbb{N}} $ is said to be shored-up if there exists a sequence of positive numbers $ (c_{k})_{k \in \mathbb{N}} $ such that $ c_{k} \rightarrow 0 $ satisfying $ \inf_{k} \sigma_{1}^{k}(c_{k}) > 0 $ for every $ k \in \mathbb{N} $.
\end{definition}

\subsection{Auxiliary lemmas}
In this subsection, we begin by presenting several key results that will be used later.

In what follows, we establish a technical lemma ensuring the existence of {\it shored up}-continuous module sequence $ (\sigma_{1}^{k})_{k\in \mathbb{N}} $, see \cite[Lemma 1]{APPT22}.

\begin{lemma}\label{Au:prop1}
Let $ (a_{j})_{j\in\mathbb{N}} \in \ell^{1}  $ and take arbitrary $ \epsilon, \delta >0 $. There exists a sequence $ (c_{j})_{j\in\mathbb{N}} \in c_{0}  $, with $ \max_{j\in\mathbb{N}}|c_{j}| \leq \epsilon^{-1}     $, such that
\begin{equation*}
\bigg( \frac{a_{j}}{c_{j}}    \bigg)_{j\in\mathbb{N}} \in \ell^{1},
\end{equation*}
and
\begin{equation*}
\epsilon \bigg( 1-\frac{\delta}{2}    \bigg)||(a_{j})||_{\ell^{1}} \leq \bigg|\bigg|\bigg( \frac{a_{j}}{c_{j}}    \bigg)\bigg|\bigg|_{\ell^{1}} \leq \epsilon(1+\delta)||(a_{j})||_{\ell^{1}}.
\end{equation*}
\end{lemma}

Before proceeding, we shall provide a criterion of non-collapsing moduli of continuity $ \boldsymbol{\mathscr{Q}} $, which will be used directly in Section \ref{Subsection4.1}, see \cite[Proposition 5]{APPT22}.

\begin{lemma}\label{Au:lem23}
  If a sequence of moduli of continuity $(\sigma_{1}^{k})_{k\in \mathbb{N}} $ is shored-up, then $ \Gamma := \cup_{k \in \mathbb{N}} \{\sigma_{1}^{k} \} $ is non-collapsing.
\end{lemma}

\begin{remark}\label{Rk26}
Dini condition can also be characterized in terms of the summability of $ \sigma_{2}^{-1} $ evaluated along geometric sequences. Namely, $ \sigma_{2}^{-1} $ satisfies the Dini condition if and only if
\begin{equation*}
\sum_{k=1}^{\infty} \sigma_{2}^{-1}(\theta^{k}) < \infty ,
\end{equation*}
for every $ \theta \in (0,1) $. This elementary quality is frequently used in the proof of Theorem \ref{Thm1}.
\end{remark}

The following lemma plays a crucial role in the proof of Theorem \ref{Thm1}. To prove this, one can follow the lines of proof of \cite[Proposition 6]{APPT22}.

\begin{lemma}[{Stability}]
\label{Se2:lemma5}
Let $ \{g_{j}\}_{j \in \mathbb{N}} $ be a sequence of Lipschitz continuous functions such that $ g_{j} \rightarrow g_{\infty} $. Assume that $\boldsymbol{\mathscr{Q}}$ is a collection of non-collapsing moduli of continuity, and $ u \in C^{0}(B_{1} \cap \{y_{n}>\varphi(y')\})$ is a bounded viscosity solution to
\begin{equation*}
\left\{
     \begin{alignedat}{2}
       \big[\sigma_{1,j}(|Du_{j}+\xi_{j}|)+a_{j}(y)\sigma_{2,j}(|Du_{j}+\xi_{j}|)\big] F_{j}(D^{2}u) & = f_{j}        \quad  &&   \text{in} \ \ B_{1} \cap \{y_{n} > \varphi(y')\}    ,    \\
          u_{j}(y)  & = g_{j}(y)    \quad  &&   \text{on}   \ \  B_{1}  \cap \{y_{n} = \varphi(y')\},        \\
     \end{alignedat}
     \right.
\end{equation*}
where $ \{\xi_{j}\}_{j \in \mathbb{N}} \subset \mathbb{R}^{n} $, $ \{f_{j}\}_{j \in \mathbb{N}} \subset C^{0}(B_{1} \cap \{y_{n} > \varphi(y')\})$, $ \sigma_{1,j}(\cdot) $, $ \sigma_{2,j}(\cdot) \in \mathscr{Q} $ and $ \{F_{j}\}_{j \in \mathbb{N}} \subset C^{0}(\mathrm{Sym}(n), \mathbb{R}) $ is uniformly $ (\lambda, \Lambda) $ elliptic. Suppose further that $ \xi_{j} \rightarrow \xi_{\infty} $, $ f_{j} \rightarrow 0 $, and $ F_{j} \rightarrow F_{\infty} $. Then there exists a subsequence from $ \{u_{j}\}_{j \in \mathbb{N}} $ which converges uniformly to $ u_{\infty} $ on $ \overline{B_{r} \cap \{y_{n} > \varphi(y')\} } $ for any $ 0 <r <1 $. Furthermore, the limit $ u_{\infty} $ solves
\begin{equation*}
\left\{
     \begin{alignedat}{2}
       F(D^{2}u_{\infty}) & = 0       \quad  &&  \text{in} \ \ B_{r} \cap \{y_{n} > \varphi(y')\}    ,        \\
          u_{\infty}(y)  & = g_{\infty}(y)    \quad  && \text{on}   \ \  B_{r}  \cap \{y_{n} = \varphi(y')\}.        \\
     \end{alignedat}
     \right.
\end{equation*}
\end{lemma}

We finish this section by providing the interior regularity results shown in \cite[Theorem 1]{AN25}.

\begin{theorem}
\label{Se2:thm3}
Let $ u \in C^{0}(\Omega) $ be a normalized viscosity solution to
\begin{equation*}
  [\sigma_{1}(|Du|)+a(x)\sigma_{2}(|Du|)] F(D^{2}u) = f        \quad    \text{in}  \ \ \Omega,
\end{equation*}
where $ 0 \leq a(\cdot) \in C^{0}(\Omega)   $ and $ \sigma_{1}(\cdot), \sigma_{2}(\cdot)$ are moduli of continuity with inverses $ \sigma_{1}^{-1}, \sigma_{2}^{-1} $. Suppose $ F $ is a uniformly $ (\lambda, \Lambda) $ elliptic, $ \sigma_{2} = o(\sigma_{1}(t)) $, $ \sigma_{2}^{-1} $ is Dini continuous, and $ f \in L^{\infty}(\Omega) $. Then $ C^{1}_{\mathrm{loc}}(\Omega)$ and there exists a modulus of continuity $ \omega: [0,+\infty) \rightarrow [0,+\infty) $, depending only on $ n, \lambda, \Lambda, \sigma_{1}, \sigma_{2} $ and $ ||f||_{L^{\infty}(\Omega)}$ such that
\begin{equation*}
  |Du(x) - Du(y)| \leq \omega(|x-y|)
\end{equation*}
for every $ x,y \in \Omega' \Subset \Omega $.
\end{theorem}

\section{Local Lipschitz estimates up to the boundary}\label{Section 3}
In this section, to obtain further Lipschitz estimate up to the boundary as in \cite{BBLL24, BSRR23}, we shall consider a bounded viscosity solution to
\begin{equation}
\label{Section3:eq1}
\left\{
     \begin{alignedat}{2}
       \big[\sigma_{1}(|Du+\xi|)+a(y)\sigma_{2}(|Du+\xi|)\big] F(D^{2}u) & = f        \quad  &&  \text{in} \ \ B_{1} \cap \{y_{n} > \varphi(y')\}    ,    \\
          u(y)  & = g(y)    \quad  &&   \text{on}   \ \  B_{1}  \cap \{y_{n} = \varphi(y')\},        \\
     \end{alignedat}
     \right.
\end{equation}
with $ \varphi $ as defined in Section \ref{Intro}.

\begin{remark}[{Smallness regime}]
\label{Se3:re1}
In the proof of Theorem \ref{Thm1}, we require
\begin{equation*}
  ||u||_{L^{\infty}(B_{1}\cap \{y_{n}>\varphi(y')\})} \leq 1, \quad ||g||_{C^{1,\alpha}(B_{1} \cap \{y_{n} = \varphi(y')\} )} \leq 1, \quad \text{and} \quad ||f||_{L^{\infty}(B_{1}\cap \{y_{n}>\varphi(y')\})} \leq \epsilon_{0}
\end{equation*}
for some small enough constant $ \epsilon_{0} \in (0,1) $. In effect, for a fixed ball $ B_{r}(x) \subset B_{1} $, we set $ \widetilde{u}: B_{1} \cap \{y_{n}>\widetilde{\varphi}(y')\} \rightarrow \mathbb{R} $ by
\begin{equation*}
  \widetilde{u}(y):=\frac{u(ry+x)}{K}
\end{equation*}
for a function $ \widetilde{\varphi} $ and $ K \geq 1 $ to be determined later. A straightforward computation shows that $ \widetilde{u} $ is a viscosity solution to
\begin{equation}
\label{Section3:eq2}
\left\{
     \begin{alignedat}{2}
       \big[\widetilde{\sigma_{1}}(|D\widetilde{u}+\widetilde{\xi}|)+\widetilde{a}(y)\widetilde{\sigma_{2}}(|D\widetilde{u}+\widetilde{\xi}|)\big] \widetilde{F}(D^{2}\widetilde{u}) & = \widetilde{f}        \quad && \text{in} \ \ B_{1} \cap \{y_{n} > \widetilde{\varphi}(y')\}    ,    \\
          \widetilde{u} & = \widetilde{g}    \quad  &&  \text{on}   \ \  B_{1}  \cap \{y_{n} = \widetilde{\varphi}(y')\}        \\
     \end{alignedat}
     \right.
\end{equation}
where
\begin{align*}
&   \widetilde{\sigma_{i}}(t):= \sigma_{i}\bigg(\frac{K}{r}t\bigg), i=1,2; \quad \widetilde{\xi}:= \frac{r}{K}\xi; \quad \widetilde{f}(y):= \frac{r^{2}}{K} f(ry+x); \quad \widetilde{a}(y):=a(ry+x);    \\
& \widetilde{F}(\text{M}):= \frac{r^{2}}{K}     F\bigg(\frac{K}{r^{2}}\text{M}\bigg); \quad \widetilde{g}(y):=\frac{g(ry+x)}{K}; \quad  \text{and}  \quad \widetilde{\varphi}(y'):=\frac{\varphi(ry'+x')-x_{n}}{r}.
\end{align*}
Note that if $ \sigma_{1} $ and $ \sigma_{2} $ admit inverse, then it can be seen that
\begin{equation*}
  \widetilde{\sigma_{i}}^{-1}(t):= \frac{r}{K} \sigma_{i}^{-1}(t), \quad  i=1,2.
\end{equation*}
In fact, $$  \widetilde{\sigma_{i}}^{-1}(\widetilde{\sigma_{i}}(t)) = \frac{r}{K} \sigma_{i}^{-1}(\widetilde{\sigma_{i}}(t)) = \frac{r}{K}\sigma_{i}^{-1}(\sigma_{i}\big(\frac{K}{r} t \big)) = t .     $$
Furthermore, by choosing $ r < K $, we readily obtain
\begin{equation*}
\left\{
     \begin{aligned}
  & \int_{0}^{\tau} \frac{\widetilde{\sigma_{2}}^{-1}(t)}{t} dt \leq \int_{0}^{\tau} \frac{\sigma_{2}^{-1}(t)}{t} dt < \infty;  \\
  & \widetilde{\sigma_{1}}(1) = \sigma_{1} \bigg( \frac{K}{r}  \bigg) \geq \sigma_{1}(1) \geq \sigma_{2}(1) \geq 1;   \\
  & \lim_{t \rightarrow 0} \widetilde{\sigma_{i}}(t) = 0, \quad i=1,2,  \quad  \text{and}  \quad  0< \widetilde{a}(y) \in C^{0}(B_{1} \cap \{y_{n} > \widetilde{\varphi}(y')\}).
  \end{aligned}
     \right.
\end{equation*}
Here we choose
\begin{equation*}
r:= \epsilon_{0}^{\frac{1}{2}}  \quad \text{and} \quad K:= 2(1+||u||_{L^{\infty}(\Omega)}+||g||_{C^{1,\alpha}(\partial \Omega)}+||f||_{L^{\infty}(\Omega)}),
\end{equation*}
then it is immediate from the choice of $ r $ and $ K $ that $ ||D^{2}\widetilde{\varphi}||_{L^{\infty}(B_{1} \cap \{y_{n} > \widetilde{\varphi}(y')\})} \leq ||D^{2}\varphi||_{L^{\infty}(B_{1} \cap \{y_{n} > \varphi(y')\})}$, and
\begin{equation*}
  ||\widetilde{g}||_{C^{1,\alpha}(B_{1}  \cap \{y_{n} = \widetilde{\varphi}(y')\})} \leq \frac{1}{K}||g||_{C^{1,\alpha}(\partial \Omega)} \leq ||g||_{C^{1,\alpha}(\partial \Omega)}.
\end{equation*}
Hence, in what follows, we will always assume that the conditions
\begin{equation*}
  ||u||_{L^{\infty}(B_{1}\cap \{y_{n}>\varphi(y')\})} \leq 1, \quad ||g||_{C^{1,\alpha}(B_{1} \cap \{y_{n} = \varphi(y')\} )} \leq 1, \quad \text{and} \quad ||f||_{L^{\infty}(B_{1}\cap \{y_{n}>\varphi(y')\})} \leq \epsilon_{0}
\end{equation*}
hold.
\end{remark}

The following lemma describes the boundary behavior of a viscosity solution $ u $ in terms of a distance function $ d $.

\begin{lemma}
\label{Se3:lemma1}
Let $ d $ be the distance to the hypersurface $ \{y_{n} = \varphi(y')\} $, and $ g \in C^{1,\alpha}(\partial \Omega) $. Then for every $ r \in (0,1) $ and $ \gamma \in (0,1) $, there exists $ \delta_{0}>0 $, depending only on $ ||f||_{L^{\infty}(B_{1}\cap \{y_{n}>\varphi(y')\})}, \lambda, \Lambda, \Omega, r, \sigma_{1}, \sigma_{2} $ and $ ||g||_{C^{1,\alpha}(B_{1} \cap \{y_{n} = \varphi(y')\} )} $, such that for every $ 0< \delta < \delta_{0} $, if $ u \in C^{0}(B_{1} \cap \{y_{n}>\varphi(y')\}) $ is a viscosity solution to \eqref{Section3:eq1} ($ \xi =0 $) with $ ||u||_{L^{\infty}(B_{1}\cap \{y_{n}>\varphi(y')\})} \leq 1 $, then
\begin{equation*}
  |u(y', y_{n})-g(y', \varphi(y'))| \leq \frac{8}{\delta} \frac{d(y)}{1+d(y)^{\gamma}} \quad \text{in} \quad B_{r} \cap \{y_{n} > \varphi(y')\}.     
\end{equation*}    
\end{lemma}
This lemma can be proved similar to \cite[Lemma 4.3]{BBLL24} by the barrier function method and comparison principle. For the sake of brevity, we omit its proof.

\begin{theorem}[{H\"{o}lder estimates for $ \xi =0 $}]
\label{Se4:thm3}
Let $ g $ be Lipschitz continuous datum. Suppose that $ u \in C^{0}(B_{1} \cap \{y_{n}>\varphi(y')\}) $ satisfies \eqref{Section3:eq1}. Then for every $ r, \gamma\in (0,1) $, $ u \in C^{0,\gamma}(B_{r} \cap \{y_{n}>\varphi(y')\}) $. Furthermore, we have the following estimate:
\begin{equation*}
  ||u||_{C^{0,\gamma}(B_{r} \cap \{y_{n}>\varphi(y')\})}  \leq C(\lambda, \Lambda, n, r, \sigma_{1}, \sigma_{2}, ||f||_{L^{\infty}(\Omega)}, ||a||_{L^{\infty}(\Omega)}, \mathrm{Lip}_{g}(\partial \Omega)).
\end{equation*}
\end{theorem}

\begin{proof}
Let $ r_{1} \in (r,1) $ be fixed. For $ x_{0} \in B_{r} \cap \{y_{n}>\varphi(y')\} $. We define
\begin{equation}\label{Section 3:eq3}
  \Phi(x,y):= u(x) - u(y) - L_{1}|x-y|^{\gamma}  - L_{2}(|x-x_{0}|^{2}+|y-x_{0}|^{2}).
\end{equation}

We claim that for large enough $ L_{1}, L_{2} $, it holds
\begin{equation}\label{Section 3:eq4}
   \Phi(x,y)  \leq 0  \quad   \text{in}  \quad (B_{r_{1}}\cap \overline{\Omega}) \times (B_{r_{1}}\cap \overline{\Omega}).
\end{equation}
This claim implies the desired H\"{o}lder estimate of $ u $.

Firstly, suppose that $ y \in B_{r_{1}} \cap \{y_{n} = \varphi(y')\} $, by using Lemma \ref{Se3:lemma1} and the idea from \cite[Theorem 4.4]{BBLL24}, we can prove that $ \Phi(x,y)  \leq 0 $.

We now show \eqref{Section 3:eq4} by contradiction, suppose that there exists $ (\overline{x}, \overline{y}) \in (B_{r_{1}}\cap \overline{\Omega}) \times (B_{r_{1}}\cap \overline{\Omega}) $ such that
\begin{equation}\label{Section 3:eq5}
  \Phi(\overline{x}, \overline{y}):= \max_{(B_{r_{1}}\cap \overline{\Omega})\times (B_{r_{1}}\cap \overline{\Omega})} \Phi(x, y) > 0.
\end{equation}
By choosing $ L_{2} > \max \{ \frac{4}{(r_{1}-r)^{2}}, \frac{1}{2(r+r_{1})}\} $, then it is easy to see that $ \overline{x} \neq \overline{y} $ (otherwise, the claim is obviously true), $ (\overline{x}, \overline{y}) \in (B_{r_{1}} \cap \{y_{n}>\varphi(y')\})^{2} $ and $ (\overline{x}, \overline{y}) \in B_{\frac{r_{1}+r}{2}} \times B_{\frac{r_{1}+r}{2}} $. Thus, by Ishii-Lions lemma \cite[Theorem 3.2]{CIL92}, it follows that for every $ \epsilon > 0 $ small enough, there exist $\mathrm{X}, \mathrm{Y} \in \mathrm{Sym}(n) $ such that
\begin{equation}\label{Section 3:eq6}
  (\gamma L_{1}(\overline{x}-\overline{y})|\overline{x}-\overline{y}|^{\gamma-2}+2L_{2}(\overline{x}-x_{0}),\mathrm{X})\in \overline{J^{2,+}}u(\overline{x})  \end{equation}
  and
\begin{equation}\label{Section 3:eq7}
  (\gamma L_{1}(\overline{x}-\overline{y})|\overline{x}-\overline{y}|^{\gamma-2}-2L_{2}(\overline{y}-x_{0}),\mathrm{-Y})\in \overline{J^{2,-}}u(\overline{y})
  \end{equation}
with
\begin{equation}\label{Section 3:eq8}
\begin{pmatrix}
\mathrm{X}  &   0   \\
0  &    \mathrm{Y}
\end{pmatrix}
\leq
\begin{pmatrix}
\mathrm{Z}   &   -\mathrm{Z}   \\
-\mathrm{Z}  &    \mathrm{Z}
\end{pmatrix}
 + (2L_{2}+\epsilon)\textbf{I}\rm{d}_{n}, \ \  0 < \epsilon \ll 1.
\end{equation}

Before two viscosity inequalities are given, for simplicity, we denote
\begin{align}
\label{Section 3:eq9}
\begin{split}
& q_{\overline{x}}:= \gamma L_{1}(\overline{x}-\overline{y})|\overline{x}-\overline{y}|^{\gamma-2}+2L_{2}(\overline{x}-x_{0}) \\
& q_{\overline{y}}:= \gamma L_{1}(\overline{x}-\overline{y})|\overline{x}-\overline{y}|^{\gamma-2}-2L_{2}(\overline{y}-x_{0}).
\end{split}
\end{align}
Then from \eqref{Section 3:eq6} and \eqref{Section 3:eq7}, we have the following two viscosity inequalities:
 \begin{equation}\label{Section 3:eq10}
 \big[\sigma_{1}(|q_{\overline{x}}|)+a(\overline{x})\sigma_{2}(|q_{\overline{x}}|)\big] F(\mathrm{X}) \geq  f(\overline{x}) \geq -||f||_{\infty}
 \end{equation}
and
 \begin{equation}\label{Section 3:eq11}
 \big[\sigma_{1}(|q_{\overline{y}}|)+a(\overline{y})\sigma_{2}(|q_{\overline{y}}|)\big] F(\mathrm{-Y}) \leq  f(\overline{y}) \leq ||f||_{\infty}.
 \end{equation}
From \eqref{Section 3:eq8}, we can easily prove that
\begin{equation*}
  \mathscr{P}^{+}_{\lambda,\Lambda}(\mathrm{X}+\mathrm{Y}) \leq \lambda [4L_{1}\gamma(\gamma-1)|\overline{x}-\overline{y}|^{\gamma-2}+6L_{2}]+ 6L_{2}\Lambda(n-1),
  \end{equation*}
then by using \hyperref[A1]{\bf (A1)}, it yields
\begin{equation}\label{Section 3:eq12}
  F(\mathrm{X}) - F(\mathrm{-Y}) \leq  \mathscr{P}^{+}_{\lambda,\Lambda}(\mathrm{X}+\mathrm{Y})  \leq \lambda [4L_{1}\gamma(\gamma-1)|\overline{x}-\overline{y}|^{\gamma-2}+6L_{2}]+ 6L_{2}\Lambda(n-1).
\end{equation}
Noticed that from the choice of $ L_{1}, L_{2} $ and \eqref{Section 3:eq9}, one obtains $ |q_{\overline{x}}|\geq 1 $ and $ |q_{\overline{y}}|\geq 1 $, then
\begin{align}
\label{Section 3:eq13}
\begin{split}
& \big[\sigma_{1}(|q_{\overline{x}}|)+a(\overline{x})\sigma_{2}(|q_{\overline{x}}|)\big] \geq  \sigma_{1}(|q_{\overline{x}}|) \geq \sigma_{1}(1) \overset{\hyperref[A5a]{\bf (A5a)}}{\geq} 1      \\
& \big[\sigma_{1}(|q_{\overline{y}}|)+a(\overline{y})\sigma_{2}(|q_{\overline{y}}|)\big] \geq \sigma_{1}(|q_{\overline{y}}|) \geq \sigma_{1}(1) \overset{\hyperref[A5a]{\bf (A5a)}}{\geq}  1.
\end{split}
\end{align}

Now we combine \eqref{Section 3:eq10}--\eqref{Section 3:eq13} to get that
\begin{equation*}
  4 \lambda L_{1} \gamma(1-\gamma)|\overline{x}-\overline{y}|^{\gamma-2} \leq 2||f||_{\infty} + 6L_{2}[\lambda+\Lambda(n-1)],
\end{equation*}
which is a contradiction, provided $ L_{1} $ is large enough.
\end{proof}

\begin{remark}\label{Section3:rmk2}
If we define the auxiliary function
 \begin{equation*}
  \Phi(x,y):= u(x) - u(y) - L_{1} \omega(|x-y|)  - L_{2}(|x-x_{0}|^{2}+|y-x_{0}|^{2}).
\end{equation*}
where $ \omega(s):= s- \omega_{0}s^{\frac{3}{2}} $ if $ s \leq s_{0} = (\frac{2}{\omega_{0}})^{2} $ and $ \omega(s) = \omega(s_{0}) $ if $ s \geq s_{0} $. By following the same idea of Theorem~\ref{Se4:thm3}, one can see that $ C^{0, 1}(B_{r} \cap \{y_{n}>\varphi(y')\}) $ for both the degenerate and singular case.
\end{remark}

\begin{theorem}[{Lipschitz estimates for large $ |\xi| $}]
\label{Se4:thm4}
 Let $ g $ be a Lipschitz function on $ \partial \Omega $ and $ \xi \in \mathbb{R}^{n} $. Assume that $ u $ is a viscosity solution to \eqref{Section3:eq1} with
 \begin{equation*}
   ||u||_{L^{\infty}(B_{1}\cap \{y_{n}>\varphi(y')\})} \leq 1 \quad \text{and} \quad  ||f||_{L^{\infty}(B_{1}\cap \{y_{n}>\varphi(y')\})} \leq \epsilon_{0}.
 \end{equation*}
Then for all $ r \in (0,1) $, there exists $ A_{0} > 0 $, such that if $ |\xi| > A_{0} $, we have that $ u \in C^{0,1}(B_{1} \cap \{y_{n}>\varphi(y')\}) $ with the estimate
\begin{equation*}
  ||u||_{C^{0,1}(B_{1} \cap \{y_{n}>\varphi(y')\})}  \leq C(\lambda, \Lambda, n, r, \sigma_{1}, \sigma_{2}, ||f||_{L^{\infty}(B_{1}\cap \{y_{n}>\varphi(y')\})}, ||a||_{L^{\infty}(\Omega)}, \mathrm{Lip}_{g}(\partial \Omega)).
\end{equation*}
\end{theorem}

Before proving Theorem \ref{Se4:thm4}, we require the following auxiliary lemma.

\begin{lemma}\label{Se4:lemma4}
Let $ g $ be $ C^{1,\alpha} $ function on $ \partial \Omega $. For every $ r, \gamma \in (0,1) $, there exists $ \delta $, depending only on $ \lambda, \Lambda, r, \sigma_{1}, \sigma_{2} $, $ ||f||_{L^{\infty}(\Omega)} $ and $ Lip_{g}(\partial \Omega) $, such that for $ b < \frac{\delta}{3} $, any viscosity $ u $ of
\begin{equation*}
\left\{
     \begin{alignedat}{2}
       \big[\sigma_{1}(|bDu+e_{n}|)+a(y)\sigma_{2}(|bDu+e_{n}|)\big] F(D^{2}u) & = f        \quad  &&  \text{in} \ \ B_{1} \cap \{y_{n} > \varphi(y')\}    ,    \\
          u(y)  & = g(y)    \quad  &&   \text{on}   \ \  B_{1}  \cap \{y_{n} = \varphi(y')\}        \\
     \end{alignedat}
     \right.
\end{equation*}
with
\begin{equation*}
  ||u||_{L^{\infty}(B_{1}\cap \{y_{n}>\varphi(y')\})} \leq 1, \quad  \text{and}  \quad ||f||_{L^{\infty}(B_{1}\cap \{y_{n}>\varphi(y')\})} \leq \epsilon_{0}
\end{equation*}
satisfies
\begin{equation*}
  |u(y', y_{n})-g(y', \varphi(y'))| \leq \frac{8}{\delta} \frac{d(y)}{1+d(y)^{\gamma}} \quad \text{in} \quad B_{r} \cap \{y_{n} > \varphi(y')\}.        
\end{equation*}
\end{lemma}

The proof follows the same arguments as in Lemma \ref{Se3:lemma1}, see also \cite[Lemma 2.4]{BSRR23}. It is omitted here to avoid repetition.

We are now ready to prove Theorem~\ref{Se4:thm4}.

\begin{proof}[{\bf Proof of Theorem~\ref{Se4:thm4}}]
The proof is similar to Theorem \ref{Se4:thm3}, we focus only on the differences. Here we define a new function
\begin{equation*}
  \Phi(x,y):= u(x) - u(y) - L_{1} \omega(|x-y|)  - L_{2}(|x-x_{0}|^{2}+|y-x_{0}|^{2}),
\end{equation*}
where $ \omega(s):= s- \omega_{0}s^{\frac{3}{2}} $ if $ s \leq s_{0} = (\frac{2}{\omega_{0}})^{2} $ and $ \omega(s) = \omega(s_{0}) $ if $ s \geq s_{0} $. Our goal remains to show that \eqref{Section 3:eq4} is true.

If $ y \in B_{r_{1}} \cap \{y_{n} = \varphi(y')\} $, we observe that $ u $ is a viscosity solution to
\begin{equation}
\label{Section3:eq14}
\left\{
     \begin{alignedat}{2}
       [\overline{\sigma}_{1}(b Du+e_{n}) + a(x) \overline{\sigma}_{2}(b Du+e_{n})] F(D^{2}u) & = \widetilde{f}  \quad && \text{in} \quad B_{1} \cap \{y_{n}>\varphi(y')\},    \\
          u(y)  & = g(y)    \quad  &&   \text{on}   \ \  B_{1}  \cap \{y_{n} = \varphi(y')\}        \\
     \end{alignedat}
     \right.
\end{equation}
where
\begin{equation*}
\overline{\sigma}_{1}(t):= \frac{\sigma_{1}(|\xi|t)}{\sigma_{1}(|\xi|)}, \quad \overline{\sigma}_{2}(t):= \frac{\sigma_{2}(|\xi|t)}{\sigma_{1}(|\xi|)}, \quad \overline{f}(t):=\frac{f(t)}{\sigma_{1}(|\xi|)}, \quad \text{and}  \quad   b:= \frac{1}{|\xi|}.
\end{equation*}
In fact, if $ \psi \in C^{2}(B_{1} \cap \{y_{n}>\varphi(y')\}) $, and $ x_{0} \in B_{1} \cap \{y_{n}>\varphi(y')\} $ such that $ u - \psi $ has a local minimum at $ x_{0} $, then
\begin{align*}
 & [\overline{\sigma}_{1}(bD\psi+e_{n})+ a(x_{0})\overline{\sigma}_{2}(bD\psi+e_{n})] F(D^{2}\psi) \\
= & \sigma_{1}^{-1}(|\xi|)[\sigma_{1}(|Du+\xi|)+ a(x_{0})\sigma_{2}(|Du+\xi|)] F(D^{2}\psi)    \\
\leq  &  \sigma_{1}^{-1}(|\xi|) f(x_{0}) =  \overline{f}(x_{0}).
\end{align*}
Hence, $ u $ is a viscosity super-solution to \eqref{Section3:eq14}. Similarly, $ u $ is a viscosity sub-solution to \eqref{Section3:eq14}. In this case, we can show \eqref{Section 3:eq4} holds by applying Lemma \ref{Se4:lemma4}.

We next proceed with the contradiction argument from Theorem \ref{Se4:thm3}. The distinction emerges in the use of the definitions of limiting superjet and subjet:
\begin{equation}\label{Section3:eq15}
 \big[\sigma_{1}(|q_{\overline{x}}+\xi|)+a(\overline{x})\sigma_{2}(|q_{\overline{x}}+\xi|)\big] F(\mathrm{X}) \geq  f(\overline{x}) \geq -||f||_{\infty},
 \end{equation}
 \begin{equation}\label{Section3:eq16}
 \big[\sigma_{1}(|q_{\overline{y}}+\xi|)+a(\overline{y})\sigma_{2}(|q_{\overline{y}}+\xi|)\big] F(\mathrm{-Y}) \leq  f(\overline{y}) \leq ||f||_{\infty}.
 \end{equation}
Noticing that $ |q_{\overline{x}}|, |q_{\overline{y}}| \leq 2L_{1} \omega'(|\overline{x}-\overline{y}|) \leq 2L_{1} $ provided that $ L_{1} $ is large enough. If we choose $ A_{0} > 3L_{1} $, then we have
\begin{equation}\label{Section3:eq17}
  |q_{\overline{x}}+\xi|, |q_{\overline{y}}+\xi| \geq L_{1}.
\end{equation}

By performing a similar calculation as in \eqref{Section 3:eq12}, it yields
\begin{equation}\label{Section3:eq18}
  F(\mathrm{X}) - F(\mathrm{-Y}) \leq  -3\lambda L_{1} \omega_{0} |\overline{x}-\overline{y}|^{-1/2} + 6 [\Lambda(n-1)+\lambda] L_{2}.
\end{equation}

Now we combine \eqref{Section3:eq15}--\eqref{Section3:eq18} to obtain
\begin{equation*}
  3\lambda L_{1} \omega_{0} |\overline{x}-\overline{y}|^{-1/2} \leq 2||f||_{\infty} +  6 [\Lambda(n-1)+\lambda] L_{2}.
\end{equation*}
By choosing $ L_{1} \gg 1 $ large enough, depending $ n, L_{2}, \epsilon, ||f||_{\infty} $, we reach a contradiction.
\end{proof}

\section{Proof of Theorems \ref{Thm1} and \ref{Thm2}}\label{Section 4}

In this section, our main goal is to prove Theorems~\ref{Thm1} and \ref{Thm2}. The proof of Theorem~\ref{Thm1} consists of three steps. Firstly, we need to establish a boundary version of the approximation lemma for \eqref{Section3:eq1}, see Lemma \ref{Se4:lemma1} below. Secondly, we seek to reconstruct a new {\em shored-up} sequence of moduli of continuity $ \{\sigma_{1}^{j}(t)+a_{j}(\cdot)\sigma_{2}^{j}(t)\}_{j \in \mathbb{N}} $ in order to overcome non-collapsing property for the term $ \sigma_{1} + a(\cdot) \sigma_{2} $, see Section \ref{Subsection4.1} below for details. Finally, by means of delicate iteration and convergence analysis, we give a complete proof of Theorem~\ref{Thm1}.

Now we are in a position to state the main result of this section.

\begin{proposition}[{Pointwise $ C^{1}$ estimates}]
\label{Se4:prop1}
Suppose that the conditions \hyperref[A1]{\bf (A1)}--\hyperref[A5a]{\bf (A5a)} and \hyperref[A6]{\bf (A6)} hold. Then there exist constants $ \epsilon > 0 $, $ \rho \in (0,\frac{1}{2}) $, and constant $  C > 0 $ depending on $ n, \lambda, \Lambda $, $ ||D^{2}\varphi||_{L^{\infty}(\Omega)} $, $ ||g||_{C^{1,\alpha}(\partial \Omega)} $, $ \sigma_{1} $ and $ \sigma_{2} $, such that for a viscosity solution $ u $ of
\begin{equation*}
\left\{
     \begin{alignedat}{2}
       \big[\sigma_{1}(|Du|)+a(y)\sigma_{2}(|Du|)\big] F(D^{2}u) & = f        \quad   &&  \text{in} \ \ B_{1} \cap \{y_{n} > \varphi(y')\}    ,    \\
          u(y)  & = g(y)    \quad  &&   \text{on}   \ \  B_{1}  \cap \{y_{n} = \varphi(y')\},        \\
     \end{alignedat}
     \right.
\end{equation*}
the followings holds: if
\begin{equation*}
  ||u||_{L^{\infty}(B_{1}\cap \{y_{n}>\varphi(y')\})} \leq 1, \quad \text{and}  \quad  ||f||_{L^{\infty}(B_{1}\cap \{y_{n}>\varphi(y')\})} \leq \epsilon,
\end{equation*}
then there exists an affine function $ l(y):= a+ b\cdot y  $ with $ |a|+ |b| \leq C $, and a modulus of continuity $ \gamma $ such that for each $ 0 < r \leq \rho $,
\begin{equation*}
  ||u-l||_{L^{\infty}({B_{r}\cap \{y_{n}>\varphi(y')\}})} \leq \gamma(r) r.
\end{equation*}
\end{proposition}

To show Proposition \ref{Se4:prop1}, we require the following approximation lemma.

\begin{lemma}
\label{Se4:lemma1}
Let $\boldsymbol{\mathscr{Q}}$  be a collection of non-collapsing moduli of continuity and $ u \in C^{0}(B_{1} \cap \{y_{n}>\varphi(y')\})$ be a viscosity solution to
\begin{equation*}
\left\{
     \begin{alignedat}{2}
       \big[\sigma_{1}(|Du+\xi|)+a(y)\sigma_{2}(|Du+\xi|)\big] F(D^{2}u) & = f        \quad && \text{in} \ \ B_{1} \cap \{y_{n} > \varphi(y')\}    ,    \\
          u(y)  & = g(y)    \quad   &&  \text{on}   \ \  B_{1}  \cap \{y_{n} = \varphi(y')\},        \\
     \end{alignedat}
     \right.
\end{equation*}
where $ \xi \in \mathbb{R}^{n} $ is an arbitrary vector and $ \sigma_{1}(\cdot), \sigma_{2}(\cdot) \in \mathscr{Q} $, satisfying
\begin{equation*}
  ||u||_{L^{\infty}(B_{1}\cap \{y_{n}>\varphi(y')\})} \leq 1, \quad \text{and}  \quad  ||g||_{C^{1,\alpha}(B_{1} \cap \{y_{n} = \varphi(y')\} )} \leq 1.
\end{equation*}
Suppose that the conditions \hyperref[A1]{\bf (A1)}--\hyperref[A5a]{\bf (A5a)} hold. Then for any $ \mu > 0 $, there exists a constant $ \epsilon >0 $ such that if
\begin{equation*}
  ||f||_{L^{\infty}(B_{1}\cap \{y_{n}>\varphi(y')\})} \leq \epsilon,
\end{equation*}
then one can find a function $ h $ of
 \begin{equation}\label{Se4:eq1}
\left\{
     \begin{alignedat}{2}
       F(D^{2}h) & = 0       \quad    && \text{in} \ \ B_{\frac{4}{5}} \cap \{y_{n} > \varphi(y')\}    ,        \\
          h(y)  & = g(y)    \quad   &&   \text{on}   \ \  B_{\frac{4}{5}}  \cap \{y_{n} = \varphi(y')\},        \\
     \end{alignedat}
     \right.
\end{equation}
such that
\begin{equation*}
  ||u-h||_{L^{\infty}({B_{1/2}\cap \{y_{n}>\varphi(y')\}})} \leq \mu.
\end{equation*}
\end{lemma}

\begin{proof}
We argue by contradiction, suppose that the conclusion does not hold, then there exist sequences $ \{\sigma_{1,j}\}_{j \in \mathbb{N}}, \{\sigma_{2,j}\}_{j \in \mathbb{N}}, \{\xi_{j}\}_{j \in \mathbb{N}}, \{a_{j}\}_{j \in \mathbb{N}}, \{u_{j}\}_{j \in \mathbb{N}} $, $ \{F_{j}\}_{j \in \mathbb{N}}, \{f_{j}\}_{j \in \mathbb{N}} $, and a positive number $ \mu_{0}  $ such that, for every $ j \in \mathbb{N} $, we have

\label{D1}   $ {(D1)} $. $ u_{j}(0) = 0 $, $ F_{j}: \mathrm{Sym}(n) \rightarrow \mathbb{R}  $ is a $ (\lambda, \Lambda) $-elliptic operator;

\label{D2}    $ {(D2)} $. $ f_{j} \in C^{0}(B_{1}\cap \{y_{n}>\varphi(y')\}) $ with $ ||f_{j}||_{L^{\infty}(B_{1}\cap \{y_{n}>\varphi(y')\})} \leq  \frac{1}{j} $;

\label{D3}     $ {(D3)} $. $ a_{j} \in C^{0}(B_{1} \cap \{y_{n} > \varphi(y')\}) $ with $ a_{j}(\cdot) \geq 0 $;

\label{D4}     $ {(D4)} $. $ \sigma_{1,j}(0) = \sigma_{2,j}(0) = 0 $, $ \sigma_{1,j}(1) \geq \sigma_{2,j}(1) \geq 1 $, and if $ \sigma_{1,j}(t_{j}) + a_{j}(x) \sigma_{2,j}(t_{j}) \rightarrow 0 $, then $ t_{j} \rightarrow 0 $;

\label{D5}     $ {(D5)} $. $ u_{j} \in C^{0}(B_{1}\cap \{y_{n}>\varphi(y')\}) $ with $ ||u_{j}||_{L^{\infty}(B_{1}\cap \{y_{n}>\varphi(y')\})} \leq 1 $ solves the equation
\begin{equation*}
\left\{
     \begin{alignedat}{2}
       \big[\sigma_{1,j}(|Du_{j}+\xi_{j}|)+a_{j}(y)\sigma_{2,j}(|Du_{j}+\xi_{j}|)\big] F_{j}(D^{2}u_{j}) & = f_{j}        \quad  &&  \text{in} \ \ B_{1} \cap \{y_{n} > \varphi(y')\}    ,    \\
          u_{j}(y)  & = g_{j}(y)    \quad  && \text{on}   \ \  B_{1}  \cap \{y_{n} = \varphi(y')\},        \\
     \end{alignedat}
     \right.
\end{equation*}
with $ ||g_{j}||_{C^{1,\alpha}(B_{1} \cap \{y_{n} = \varphi(y')\} )} \leq 1 $, but
\begin{equation}\label{Se4:eq2}
||u_{j}-h||_{L^{\infty}({B_{1/2}\cap \{y_{n}>\varphi(y')\}})} \geq \mu_{0}
\end{equation}
for any $ h $ satisfying \eqref{Se4:eq1}.

The condition \hyperref[D1]{(D1)} implies that $ F_{j} \rightarrow F_{\infty} $ for some uniformly elliptic operator $ F_{\infty}$. Analogously, from  \hyperref[D5]{(D5)} it is readily seen that $ g_{j} \rightarrow g_{\infty} $ uniformly. In addition, from \hyperref[D2]{(D2)} and \hyperref[D3]{(D3)}, we can easily see that $ f_{j} $, $ a_{j} $ converge uniformly to 0 and continuous function $ a_{\infty} $, respectively. Now we distinguish two cases.

\smallskip
\noindent
{\em Case 1.} If $ \{\xi_{j}\}_{j \in \mathbb{N}} $ is unbounded sequence, then $ |\xi_{j}| \rightarrow \infty $ and up to a subsequence, we have that $ |\xi_{j}| > A_{0} $ for some $ A_{0} > 0 $. Thus we can apply Theorem \ref{Se4:thm4} for $ u_{j} $ and so by Arzel$ \grave{a}$-Ascoli theorem, it follows that $ u_{j} \rightarrow u_{\infty} $ uniformly in $ B_{r} \cap \{y_{n} > \varphi(y')\} $ for any $ 0< r <1 $. By Lemma \ref{Se2:lemma5}, we get that $ u_{\infty} $ solves
\begin{equation*}
\left\{
     \begin{alignedat}{2}
       F(D^{2}u_{\infty}) &= 0       \quad   &&  \text{in} \ \ B_{\frac{4}{5}} \cap \{y_{n} > \varphi(y')\}    ,        \\
          u_{\infty}(y)  &= g_{\infty}(y)    \quad  &&  \text{on}   \ \  B_{\frac{4}{5}}  \cap \{y_{n} = \varphi(y')\},         \\
     \end{alignedat}
     \right.
\end{equation*}
which leads to a contradiction with \eqref{Se4:eq2} provided that we choose $ h = u_{\infty} , g = g_{\infty}, F = F_{\infty}$.

\smallskip
\noindent
{\em Case 2.} If $ \{\xi_{j}\}_{j \in \mathbb{N}} $ is bounded sequence, then up to a subsequence, $ \xi_{j} \rightarrow \xi_{\infty} $ for some $ \xi_{\infty} \in \mathbb{R}^{n} $. For simplicity, we take a sequence $ \widetilde{u_{j}}(y):= u_{j}(y) + y\cdot \xi_{j}  $, it can be seen that
\begin{equation*}
\left\{
     \begin{alignedat}{2}
      \big[\sigma_{1,j}(|D\widetilde{u_{j}}|)+a_{j}(y)\sigma_{2,j}(|D\widetilde{u_{j}}|)\big] F_{j}(D^{2}\widetilde{u_{j}}) & = f_{j} \quad && \text{in} \ \ B_{1} \cap \{y_{n} > \varphi(y')\} , \\
      \widetilde{u_{j}}(y) & = \widetilde{g_{j}}(y) \quad && \text{on} \ \ B_{1} \cap \{y_{n} = \varphi(y')\} ,
     \end{alignedat}
     \right.
\end{equation*}
where $ \widetilde{g_{j}}(y):= g_{j}(y)  + y\cdot \xi_{j}  $. Similar to the analysis of the {\em Case 1}, it concludes that $ u_{\infty} $ satisfies
\begin{equation*}
\left\{
     \begin{alignedat}{2}
       F(D^{2}u_{\infty}) & = 0       \quad  &&  \text{in} \ \ B_{\frac{4}{5}} \cap \{y_{n} > \varphi(y')\}    ,        \\
          u_{\infty}(y)  & = g_{\infty}(y)    \quad  &&  \text{on}   \ \  B_{\frac{4}{5}}  \cap \{y_{n} = \varphi(y')\},        \\
     \end{alignedat}
     \right.
\end{equation*}
which gives a contradiction as well, similar to that in {\em Case 1}.
\end{proof}

\begin{remark}\label{Se4:rmk1}
Regarding the regularity of viscosity solutions to problem \eqref{Se4:eq1}, we make some the following comments.

\begin{enumerate}

\item Based on the result of \cite[Proposition 2.2]{MS06}, then $ h \in C^{1,\alpha'}(\overline{B_{\frac{2}{5}}\cap \{y_{n} > \varphi(y')\}}) $, where $ 0 < \alpha' \leq \alpha < 1 $. More recently, Ara\'{u}jo and Sirakov \cite[Theorem 1.1]{AS23} provides a more precise and quantitative relationship between $ \alpha $ and $ \alpha' $, that is, $ \alpha'= \min \{\alpha_{0}, \alpha\}  $, where $ \alpha_{0}$ is the optimal H\"{o}lder exponent for solutions to $ F$-harmonic function.

\vspace{2mm}

\item If the regularity of boundary datum $ g $ is only $ C^{1} $, then Lian-Zhang \cite[Theorem 1.11]{LZ26} established $ C^{\alpha''} $ regularity of $ h $, for any $ 0< \alpha''<1 $. In effect, the regularity of $ h $ can reach up to Log-Lipschitz in terms of \cite[Theorem 11.7]{LWZ20}, but cannot reach $ C^{0,1} $.

\end{enumerate}

Therefore, the boundary data $ g  $ is assumed to be $ C^{1,\alpha} $ in order to ensure that the solution $ h $ to \eqref{Se4:eq1} enjoys $ C^{1,\alpha'} $, thereby enabling the construction of the shored-up sequence in \cite[Section 7]{APPT22}.
\end{remark}

\subsection{The construction of the shored-up sequence}\label{Subsection4.1}
We shall define appropriate moduli of continuity to ensure $ \mu_{1} > r $. We start by introducing $ \gamma (t):= t \sigma_{2}(t)  $. Because $ t \mapsto t \sigma_{2}(t)  $ is a bijective map, it has an inverse. Let $ \widetilde{\omega} (t):= \gamma^{-1}(t)   $. We examine the choice of $ \mu_{1} $ by $ \widetilde{\omega} (t) $ as follows.

Suppose first
 $$ \frac{t^{\alpha'}}{\widetilde{\omega}(t)} \rightarrow 0, $$
then choose small $ 0 < r < \frac{1}{2} $ such that
\begin{equation*}
2Lr^{\alpha'} := \mu_{1} > r.
\end{equation*}
On the contrary, suppose
\begin{equation*}
  \frac{t^{\alpha'}}{\widetilde{\omega}(t)}   \rightarrow M_{1},
\end{equation*}
for some constant $ M_{1} >0 $, then we select small $ 0 < r < \frac{1}{2} $ and any $ 0<\widetilde{\alpha} < \alpha' $ such that
\begin{equation*}
  2Lr^{\alpha'} = r^{\widetilde{\alpha}} := \mu_{1} > r.
\end{equation*}

In what follows, we proceed by setting $  0< \theta :=   \frac{r}{\mu_{1}} <1   $ and considering the sequence
\begin{equation*}
(a_{k})_{k\in \mathbb{N}} := (\sigma_{2}^{-1}(\theta^{k}))_{k\in\mathbb{N}}.
\end{equation*}
Under the assumptions \hyperref[A5a]{\bf (A5a)} and \hyperref[A6]{\bf (A6)}, the inverse $ \sigma_{2}^{-1} $ is Dini-continuous. By Remark \ref{Rk26}, the sequence $(\sigma_{2}^{-1}(\theta^{k}))_{k\in\mathbb{N}} $ is summable, and $ (a_{k})_{k\in \mathbb{N}} \in \ell^{1} $.
Now we utilize Lemma \ref{Au:prop1}, there exists a sequence $ (c_{k})_{k\in \mathbb{N}} \in c_{0}$ such that
\begin{equation}
\frac{7}{10} \sum_{k=1}^{\infty} \sigma_{2}^{-1}(\theta^{k}) \leq \sum_{k=1}^{\infty} \frac{\sigma_{2}^{-1}(\theta^{k})}{c_{k}}  \leq \sum_{k=1}^{\infty} \sigma_{2}^{-1}(\theta^{k}).
\end{equation}

Finally, we design a sequence of moduli of continuity $ (\sigma_{1}^{k}(\cdot))_{k\in \mathbb{N}}  $ and $ (\sigma_{2}^{k}(\cdot))_{k\in \mathbb{N}}  $ given by
\begin{equation}
\left\{
     \begin{aligned}
& \sigma_{1}^{0}(t):=\sigma_{1}(t)  \quad  \text{and}  \quad   \sigma_{2}^{0}(t):=\sigma_{2}(t);   \\
& \sigma_{1}^{1}(t):=\frac{\mu_{1}}{r}\sigma_{1}(\mu_{1}t)  \quad  \text{and}  \quad  \sigma_{2}^{1}(t):=\frac{\mu_{1}}{r}\sigma_{2}(\mu_{1}t);   \\
&\sigma_{1}^{2}(t):=\frac{\mu_{1}\mu_{2}}{r^{2}}\sigma_{1}(\mu_{1}\mu_{2}t)  \quad  \text{and}  \quad   \sigma_{2}^{2}(t):=\frac{\mu_{1}\mu_{2}}{r^{2}}\sigma_{2}(\mu_{1}\mu_{2}t)           ;    \\
& \cdots \cdots \cdots   \\
& \sigma_{1}^{k}(t):=\frac{\mu_{1}\mu_{2}\cdots\mu_{k}}{r^{k}}\sigma_{1}(\mu_{1}\mu_{2}\cdots \mu_{k}t)  \quad  \text{and}  \quad  \sigma_{2}^{k}(t):=\frac{\mu_{1}\mu_{2}\cdots\mu_{k}}{r^{k}}\sigma_{2}(\mu_{1}\mu_{2}\cdots \mu_{k}t)    .
\end{aligned}
     \right.
\end{equation}

with $ \mu_{1} > r   $ as defined and $ (\mu_{k})_{k \in \mathbb{N}}  $ determined as follows. If
\begin{equation*}
\frac{\prod_{i=1}^{k}\mu_{i}}{r^{k}}\sigma_{2} \bigg(\prod_{i=1}^{k}\mu_{i} c_{k} \bigg) \geq 1,
\end{equation*}
then $ \mu_{k} = \mu_{k-1}    $. Otherwise $ \mu_{k-1} < \mu_{k} <1   $ is chosen to ensure
\begin{equation*}
\frac{\prod_{i=1}^{k}\mu_{i}}{r^{k}}\sigma_{2} \bigg(\prod_{i=1}^{k}\mu_{i} c_{k} \bigg) = 1.
\end{equation*}

It can easily be seen that a sequence of moduli of continuity $ \{\sigma_{1}^{j}(t)+a_{j}(\cdot)\sigma_{2}^{j}(t)\}_{j \in \mathbb{N}} $ is {\it shored-up}, then using Lemma \ref{Au:lem23} and \cite[Proposition 2.9]{AN25}, it follows that
$$ \boldsymbol{\mathscr{Q}}:=                       \big\{\sigma_{1}^{0}(t)+a_{0}(\cdot)\sigma_{2}^{0}(t), \sigma_{1}^{1}(t)+a_{1}(\cdot)\sigma_{2}^{1}(t), \cdots, \sigma_{1}^{k}(t)+a_{k}(\cdot)\sigma_{2}^{k}(t), \cdots \big\}     $$
is non-collapsing.

Now we commence with the proof of Proposition \ref{Se4:prop1}.

\begin{proof}[{\bf Proof of Proposition \ref{Se4:prop1}}]
From Remark \ref{Se3:re1} in Section \ref{Section 3}, without loss of generality, we may assume that
\begin{equation*}
  ||u||_{L^{\infty}(B_{1}\cap \{y_{n}>\varphi(y')\})} \leq 1, ||g||_{C^{1,\alpha}(B_{1} \cap \{y_{n} = \varphi(y')\} )} \leq 1, \quad \text{and} \quad ||f||_{L^{\infty}(B_{1}\cap \{y_{n}>\varphi(y')\})} \leq \epsilon_{0}.
\end{equation*}
We make use of a recursive argument inspired in \cite[Lemma 5.1]{BBLL24}, see also \cite[Lemma 4.3]{BSRR23}. Here we make the following claim:

{\em Claim}. There exist universal constants $ 0 < r \ll 1 $, $ C > 1 $, and a sequence of affine functions
\begin{equation*}
  l_{k}(y):=A_{k}  +  B_{k} \cdot y,
\end{equation*}
where $ \{A_{k}\}_{k \in \mathbb{N}} \subset \mathbb{R} $ and $ \{B_{k}\}_{k \in \mathbb{N}} \subset \mathbb{R}^{n} $ satisfy, for every $ k \in \mathbb{N}$,

\label{i}    $ \text{(i)} $. $ \sup_{y \in B_{r^{k}}\cap \{y_{n} >\varphi(y')\}}|u(y)-l_{k}(y)| \leq \bigg(\prod_{i=1}^{k} \mu_{i}\bigg) r^{k} $;

\label{ii}   $ \text{(ii)} $. $ |A_{k}-A_{k-1}| \leq C \bigg(\prod_{i=1}^{k-1} \mu_{i} \bigg) r^{k-1} $, and $ |B_{k}-B_{k-1}| \leq C \bigg( \prod_{i=1}^{k-1} \mu_{i} \bigg) $.

{\em Step 1}. First we verify the case of $ k=1 $. From Lemma \ref{Se4:lemma1}, we have that for $ r $ small,
\begin{equation}\label{Se4:eq7}
  ||u-h||_{L^{\infty}({B_{r}\cap \{y_{n}>\varphi(y')\}})} \leq \mu,
\end{equation}
for some $ \mu > 0 $, to be determined later. By pointwise boundary regularity for \eqref{Se4:eq1}, see Remark \ref{Se4:rmk1}, there exist affine function $ l_{0} $ and universal constant $ C >0 $ such that
\begin{equation}\label{Se4:eq8}
  \sup_{y \in B_{r}\cap \{y_{n}>\varphi(y')\}}|h(y) - l_{0}(y)| \leq C r^{1+\alpha'},
\end{equation}
for some $ \alpha' \in (0,1) $. We combine \eqref{Se4:eq7}, \eqref{Se4:eq8} and the choice of $ \mu_{1} $ in Section \ref{Subsection4.1} to obtain
\begin{equation*}
  \sup_{y \in B_{r}\cap \{y_{n}>\varphi(y')\}} |u(y) - l_{0}(y)| \leq C r^{1+\alpha'} + \mu = \frac{1}{2}\mu_{1}r+ \mu.
\end{equation*}
Selecting $ \mu = \frac{1}{2}\mu_{1}r $, the above inequality then becomes
\begin{equation*}
 \sup_{y \in B_{r}\cap \{y_{n}>\varphi(y')\}} |u(y) - l_{0}(y)| \leq \mu_{1} r,
\end{equation*}
which completes the proof of $ k=1 $. Here $ A_{0} = B_{0} = 0 $, $ A_{1} =l_{1}(0) $ and $ B_{1} = Dl_{1}(0) $.

{\em Step 2}. Next we define a rescaled function:
\begin{equation*}
  u_{1}(y):=\frac{u(ry)-l_{0}(ry)}{\mu_{1}r},
\end{equation*}
then a routine calculation reveals that $ u_{1} $ satisfies
\begin{equation*}
\left\{
     \begin{alignedat}{2}
       \bigg[\sigma_{1}^{1}\big(|Du_{1}+\xi_{1}|\big)+a_{1}(y)\sigma_{2}^{1}\big(|Du_{1}+\xi_{1}|\big)\bigg] F_{1}(D^{2}u_{1}) & = f_{1}       \quad   &&  \text{in} \ \ B_{1} \cap \{y_{n} > \varphi_{1}(y')\}    ,    \\
          u_{1}(y)  & = g_{1}(y)    \quad  &&   \text{on}   \ \  B_{1}  \cap \{y_{n} = \varphi_{1}(y')\},        \\
     \end{alignedat}
     \right.
\end{equation*}
where
\begin{equation*}
\left\{
     \begin{aligned}
& F_{1}(\mathrm{M}):= \frac{r}{\mu_{1}} F\bigg(\frac{\mu_{1}}{r}\mathrm{M}\bigg), \quad  a_{1}(y):=a(ry);   \\
& \sigma_{1}^{1}(t):= \frac{\mu_{1}}{r} \sigma_{1}(\mu_{1}t),   \quad   \sigma_{2}^{1}(t):= \frac{\mu_{1}}{r} \sigma_{2}(\mu_{1}t);   \\
& f_{1}(y):= f(ry),   \quad  g_{1}(y):= \frac{g(ry)-l_{0}(ry)}{\mu_{1}r};    \\
& \xi_{1}:= \frac{1}{\mu_{1}}Dl_{0}, \quad \text{and}  \quad \varphi_{1}(y):= r^{-1} \varphi(ry').
\end{aligned}
     \right.
\end{equation*}

It can be easily checked that
\begin{enumerate}[label=(\roman*)]
\item  $ F_{1} $ satisfies \hyperref[A1]{\bf (A1)} with the same constants $ (\lambda, \Lambda) $;

\item  $ \sigma_{1}^{1}(\cdot), \sigma_{2}^{1}(\cdot) $ satisfy \hyperref[A5a]{\bf (A5a)} and $ \sigma_{1}^{1}(\cdot), \sigma_{2}^{1}(\cdot) \in \mathscr{Q} $;

\item  By {\em Step 1}, it can be seen that $ ||u_{1}||_{L^{\infty}(B_{1}\cap \{y_{n}>\varphi_{1}(y')\})} \leq 1$;

\item  $  ||f_{1}||_{L^{\infty}(B_{1}\cap \{y_{n}>\varphi_{1}(y')\})} \leq  ||f||_{L^{\infty}(B_{1}\cap \{y_{n}>\varphi(y')\})} \leq \epsilon $ by reduction;

\item  $ ||D^{2}\varphi_{1}||_{L^{\infty}(B_{1}\cap \{y_{n}>\varphi_{1}(y')\})}  \leq r||D^{2}\varphi||_{L^{\infty}(B_{1}\cap \{y_{n}>\varphi(y')\})} \leq  ||D^{2}\varphi||_{L^{\infty}(B_{1}\cap \{y_{n}>\varphi(y')\})} $;

\item  A straightforward computation yields that for any $ y,z \in B_{1}  \cap \{y_{n} = \varphi_{1}(y')\} $
\begin{align*}
|Dg_{1}(y) - Dg_{1}(z)| & = \frac{1}{\mu_{1}} |Dg(ry) - Dg(rz)|  \leq \frac{r^{\alpha}}{\mu_{1}} ||g||_{C^{1,\alpha}(B_{1}\cap \{y_{n}>\varphi(y')\})} |y-z|^{\alpha}  \\
& \leq \frac{1}{2C}r^{\alpha-\alpha'} ||g||_{C^{1,\alpha}(B_{1}\cap \{y_{n}=\varphi(y')\})} |y-z|^{\alpha}   \\
&  \leq ||g||_{C^{1,\alpha}(B_{1}\cap \{y_{n}=\varphi(y')\})} |y-z|^{\alpha}  \quad  \text{for} \  r \ \text{small enough},
\end{align*}
then we have that
\begin{equation*}
  ||g_{1}||_{C^{1,\alpha}(B_{1}\cap \{y_{n}>\varphi_{1}(y')\})} \leq ||g||_{C^{1,\alpha}(B_{1}\cap \{y_{n}=\varphi(y')\})} \leq 1.
\end{equation*}
\end{enumerate}

Hence, we can apply Lemma \ref{Se4:lemma1} and use the same argument as in {\em Step 1} to obtain an affine function $ l_{1} $ such that
\begin{equation*}
  \sup_{y \in B_{r}\cap \{y_{n}>\varphi_{1}(y')\}} |u_{1}(y) - l_{1}(y)| \leq \mu_{1} r.
\end{equation*}
Iterating inductively the above reasoning, we define a rescaled function:
\begin{equation*}
  u_{k+1}(y):=\frac{u_{k}(ry)-l_{k}(ry)}{\mu_{k+1}r},
\end{equation*}
then simple calculations show that $ u_{k+1} $ solves
\begin{equation*}
\left\{
     \begin{alignedat}{2}
       & \bigg[\sigma_{1}^{k+1}\big(|Du_{k+1}+\xi_{k+1}|\big)+a_{k+1}(y)\sigma_{2}^{k+1}\big(|Du_{k+1}+\xi_{k+1}|\big)\bigg] F_{k+1}(D^{2}u_{k+1})   = f_{k+1}      \\
          &  \qquad  \qquad  \qquad  \qquad  \qquad  \qquad  \qquad  \qquad   \qquad  \qquad  \qquad  \qquad  \qquad \qquad    \text{in} \quad B_{1} \cap \{y_{n} >  \varphi_{k+1}(y')\},         \\
         &  u_{k+1}(y) = g_{k+1}(y)   \qquad  \qquad  \qquad  \qquad  \qquad   \qquad  \qquad  \qquad  \qquad  \qquad \text{on}   \quad   B_{1}  \cap \{y_{n} = \varphi_{k+1}(y')\},           \\
     \end{alignedat}
     \right.
\end{equation*}
where
\begin{equation*}
\left\{
     \begin{aligned}
& \sigma_{1}^{k+1}(t):= \frac{\mu_{1}\mu_{2}\cdots\mu_{k+1}}{r^{k+1}} \sigma_{1}(\mu_{1}\mu_{2}\cdots\mu_{k+1}), \quad  \sigma_{2}^{k+1}(t):= \frac{\mu_{1}\mu_{2}\cdots\mu_{k+1}}{r^{k+1}} \sigma_{2}(\mu_{1}\mu_{2}\cdots\mu_{k+1});     \\
& F_{k+1}(\mathrm{M}):= \frac{r^{k+1}}{\mu_{1}\mu_{2}\cdots\mu_{k+1}} F\big(\frac{\mu_{1}\mu_{2}\cdots\mu_{k+1}}{r^{k+1}} \mathrm{M}\big), \quad    a_{k+1}(y):= a(r^{k+1}y);      \\
& f_{k+1}(y):=f(r^{k+1}y),   \quad  g_{k+1}(y):= \frac{g(r^{k+1}y)-l_{0}(r^{k+1}y)-\mu_{1}rl_{1}(r^{k}y)-\mu_{1}\mu_{2}\cdots\mu_{k}r^{k}l_{k}(ry)}{\mu_{1}\mu_{2}\cdots\mu_{k+1}r^{k+1}};        \\
& \xi_{k+1}:=\frac{1}{\mu_{k+1}} Dl_{k}  + \frac{1}{\mu_{k}\mu_{k+1}} Dl_{k-1}  +  \frac{1}{\mu_{k-1}\mu_{k}\mu_{k+1}}Dl_{k-2}  +  \cdots \frac{1}{\mu_{1}\mu_{2}\cdots\mu_{k+1}}Dl_{0} \quad \text{and}  \\
& \varphi_{k+1}(y):= r^{-(k+1)} \varphi(r^{k+1}y').
\end{aligned}
     \right.
\end{equation*}

Analogously, we can verify the hypotheses of Lemma \ref{Se4:lemma1} for $ F_{k+1}, f_{k+1}, g_{k+1}, \xi_{k+1}, \varphi_{k+1}, a_{k+1}, \sigma_{1}^{k+1} $ and $ \sigma_{2}^{k+1} $. Therefore, we can apply Lemma \ref{Se4:lemma1} again for $ u_{k+1}$ and use the same argument as in {\em Step 1} to obtain an affine function $ l_{k+1} $ such that
\begin{equation*}
  \sup_{y \in B_{r}\cap \{y_{n}>\varphi_{k+1}(y')\}} |u_{k+1}(y) - l_{k+1}(y)| \leq \mu_{1} r.
\end{equation*}

{\em Step 3}. Scaling back to the solution $ u $, we have that
\begin{equation*}
  \sup_{y \in B_{r^{k}}\cap \{y_{n}>\varphi(y')\}} |u(y) - l_{k}(y)| \leq \bigg(\prod_{i=1}^{k}\mu_{i}\bigg)r^{k},
\end{equation*}
where $ l_{k} $ is given by
\begin{equation*}
  l_{k}(y):= l_{0}(y) + \sum_{i=1}^{k-1} l_{i}(r^{-i}y)\bigg(\prod_{j=1}^{i}\mu_{i}\bigg)r^{i}.
\end{equation*}
Writing $ l_{k}(y) $ as $ A_{k} + B_{k} \cdot y $, then it concludes that \hyperref[ii]{\text{(ii)}} in {\em Claim} also holds true. Once {\em Claim} is proved, the rest of proof of Proposition \ref{Se4:prop1} can follow \cite[Section 8]{APPT22} or \cite[Section 3.3]{WJ26}.
\end{proof}

\subsection{Proof of Theorems \ref{Thm1} and \ref{Thm2}} With these preparations, we are finally ready to present the proof of Theorems \ref{Thm1} and \ref{Thm2}.

\begin{proof}[{Proof of Theorem~\ref{Thm1}}]

The proof of Theorem~\ref{Thm1} can be completed by combining Proposition \ref{Se4:prop1}, Theorem \ref{Se2:thm3} and covering $ \partial \Omega $ with a finite number of such neighborhoods.
\end{proof}

\begin{proof}[{Proof of Theorem~\ref{Thm2}}]
From Remark \ref{Section3:rmk2}, we have that
\begin{equation*}
  ||Du||_{L^{\infty}(B_{1} \cap \{y_{n}>\varphi(y')\})} \leq C=C(\lambda, \Lambda, n, \sigma_{1}, \sigma_{2}, \epsilon_{0}, ||g||_{C^{0,1}(\partial \Omega)}).
\end{equation*}
We observed that $ u $ is a viscosity solution to
\begin{equation*}
\left\{
     \begin{alignedat}{2}
       F(D^{2}u) & = \widetilde{f}        \quad   &&  \text{in} \ \ B_{1} \cap \{y_{n} > \varphi(y')\}    ,    \\
          u(y)  & = g(y)    \quad  &&   \text{on}   \ \  B_{1}  \cap \{y_{n} = \varphi(y')\},        \\
     \end{alignedat}
     \right.
\end{equation*}
where $ \widetilde{f}(y):= \frac{f(y)}{(\sigma_{1}(|Du|)+a(y)\sigma_{2}(|Du|))} $. We note from \hyperref[A5b]{\bf (A5b)} that $ \sigma_{1}(t) \geq c_{0} $ for $ 0< t < t_{0} $, then it follows that
\begin{equation*}
  ||\widetilde{f}||_{L^{\infty}(B_{1} \cap \{y_{n} > \varphi(y')\})}  \leq \frac{2}{c_{0}} ||f||_{L^{\infty}(B_{1} \cap \{y_{n} > \varphi(y')\})}.
\end{equation*}
Applying \cite[Theorem 2.1]{AS23} for $ u $ to get $ u \in C^{1,\beta}(B_{1} \cap \{y_{n} > \varphi(y')\}) $ with the estimate
\begin{equation}\label{Se4:eq9}
\begin{aligned}
  & ||u||_{C^{1,\beta}(B_{1} \cap \{y_{n} > \varphi(y')\})} \\
   \leq & C \left(||u||_{L^{\infty}(\overline{B_{1} \cap \{y_{n} > \varphi(y')\}})}   + ||g||_{C^{1,\alpha}(\overline{B_{1} \cap \{y_{n} > \varphi(y')\}})} + ||f||_{L^{\infty}(B_{1} \cap \{y_{n} > \varphi(y')\})}\right),
\end{aligned}
\end{equation}
where $ C > 0 $ is a universal constant.

Ultimately, the proof of Theorem~\ref{Thm2} can be completed by combining \eqref{Se4:eq9}, \cite[Theorem 1.2]{BBO23} and covering $ \partial \Omega $ with a finite number of such neighborhoods.
\end{proof}

\section*{Acknowledgment}
The first author would like to thank Professor Kai Zhang (University of Granada) for sharing his insight about boundary regularity theory for fully nonlinear equations during the preparation of this manuscript. F. Jiang has been supported by the National Natural Science Foundation of China (No. 12271093) and the Jiangsu Provincial Scientific Research Center of Applied Mathematics (Grant No. BK20233002), and Shanghai Institute for Mathematics and Interdisciplinary Sciences (SIMIS) under grant number SIMIS-ID-2025-AD.


\begin{thebibliography}{GPPSVG17}
\bibliographystyle{alpha}

%
%
%
%



\bibitem[AKSZ07]{AKSZ07}
H.~Aikawa, T.~Kilpel\"{a}inen, N.~Shanmugalingam and X.~Zhong, Boundary Harnack principle for $ p$-harmonic functions in smooth Euclidean domains. \newblock {\em Potential Anal.}, \textbf{26} (2007), 281--301.




\bibitem[APPT22]{APPT22}
 P.~Andrade, D.~Pellegrino, E.~Pimentel and E.~Teixeira, $C^{1}$-regularity for degenerate diffusion equations. \newblock {\em Adv. Math.}, \textbf{409} (2022), part B, Paper No. 108667, 34 pp.




\bibitem[AN25]{AN25}
P.~Andrade and M.~Nascimento, Interior regularity estimates for fully nonlinear equations with arbitrary nonhomogeneous degeneracy laws. \newblock {\em arXiv:2501.03357.}, 2025.



\bibitem[AS23]{AS23}
D.~J.~Ara\'{u}jo and B.~Sirakov, Sharp boundary and global regularity for degenerate fully nonlinear elliptic equations. \newblock {\em J. Math. Pures Appl.}, \textbf{169} (2023), 138--154.



\bibitem[ART15]{ART15}
      D.~J.~Ara\'{u}jo, G.~Ricarte and E.~V.~Teixeira, Geometric gradient estimates for solutions to degenerate elliptic equations, \newblock {\em Calc. Var. Partial Differential Equations.},  \textbf{53} (2015), 605--625.



\bibitem[BBO23]{BBO23}
S.~Baasandorj, S.~S.~Byun and J.~Oh, $ C^{1}$ regularity for some degenerate/singular fully nonlinear elliptic equations. \newblock {\em Appl. Math. Lett.}, \textbf{146} (2023), Paper No. 108830, 10 pp.



\bibitem[BBLL24]{BBLL24}
S.~Baasandorj, S.~S.~Byun, K.~A.~Lee and S.~C.~Lee, Global regularity results for a class of singular/degenerate fully nonlinear elliptic equations. \newblock {\em Math. Z.}, \textbf{306} (2024), no. 1, Paper No. 1, 26 pp.







\bibitem[BSRR23]{BSRR23}
E.~Bezerra J\'{u}nior, J.~V.~da Silva, G.~Rampasso and G.~Ricarte, Global regularity for a class of fully nonlinear PDEs with unbalanced variable degeneracy. \newblock {\em J. Lond. Math. Soc.}, \textbf{108} (2023), 622--665.



\bibitem[BD14]{BD14}
I.~Birindelli and F.~Demengel, $ C^{1,\beta}$ regularity for Dirichlet problems associated to fully nonlinear degenerate elliptic equations. \newblock {\em ESAIM Control Optim. Calc. Var.}, \textbf{20} (2014), 1009--1024.




%




\bibitem[CC95]{CC95}
L.~Caffarelli and X.~Cabr\'{e}, \newblock {\em Fully nonlinear elliptic equations}, volume 43. American Mathematical Society, 1995.





\bibitem[CJ17]{CJ17}
T.~W.~Chaundy and A.~E.~Jolliffe,
The uniform convergence of a certain class of trigonometrical series.
\newblock {\em Proc. London Math. Soc. (2)}, \textbf{15} (1917), 214--216.



\bibitem[CIL92]{CIL92}
M.~G.~Crandall, H.~Ishii and P.~L.~Lions, User's guide to viscosity solutions of second order partial differential equations. \newblock {\em Bull. Amer. Math. Soc. (N.S.).}, \textbf{27} (1992), 1--67.



\bibitem[De21]{De21}
C.~De Filippis, Regularity for solutions of fully nonlinear elliptic equations with nonhomogeneous degeneracy. \newblock {\em Proc. Roy. Soc. Edinburgh Sect. A.}, \textbf{151} (2021), 110--132.


\bibitem[GMV21]{GMV21}
H.~Gissy, S.~Miihkinen and J.~A.~Virtanen,
On the exponential integrability of conjugate functions.
\newblock {\em J. Fourier Anal. Appl.}, \textbf{27} (2021),
Paper No. 87, 17 pp.



\bibitem[HH21]{HH21}
P.~Harjulehto and P.~H\"{a}st\"{o}, Double phase image restoration. \newblock {\em J. Math. Anal. Appl.}, \textbf{501} (2021), no. 1, Paper No. 123832, 12 pp.





\bibitem[IS13]{IS13}
        C.~Imbert and L.~Silvestre, $ C^{1,\alpha} $ regularity of solutions of some degenerate fully non-linear elliptic equations. \newblock {\em Adv. Math.}, \textbf{233} (2013), 196--206.





\bibitem[LL23]{LL23}
D.~Li and X.~Li, Pointwise boundary $ C^{1,\alpha}$ estimates for some degenerate fully nonlinear elliptic equations on $ C^{1,\alpha}$-domains, \newblock {\em arXiv:2305.12847.}, 2023.



\bibitem[LZ26]{LZ26}
Y.~Lian and K.~Zhang, Boundary H\"{o}lder regularity for elliptic equations on Reifenberg flat domains, \newblock {\em  Manuscripta Math.}, \textbf{177} (2026), Paper No. 14, 21 pp.


\bibitem[LWZ20]{LWZ20}
Y.~Lian, L.~Wang and K.~Zhang, Pointwise regularity for fully nonlinear elliptic equations in general forms, \newblock {\em 	arXiv:2012.00324.}, 2020.




\bibitem[MS06]{MS06}
 E.~Milakis and L.~Silvestre, Regularity for fully nonlinear elliptic equations with Neumann boundary data. \newblock {\em Comm. Partial Differential Equations.}, \textbf{31} (2006), 1227--1252.




\bibitem[WJ26]{WJ26}
     J.~Wang and F.~Jiang, Regularity of solutions for degenerate or singular fully nonlinear integro-differential equation. \newblock {\em Commun. Contemp. Math.},  \textbf{28} (2026), Paper No. 2550080, 36 pp.







\bibitem[Z93]{Z93}
V.~Zhikov, Lavrentiev phenomenon and homogenization for some variational problems. \newblock {\em C. R. Acad. Sci. Paris S\'{e}r. I Math.}, \textbf{316} (1993), 435--439.



\bibitem[Z02]{Z02}
A.~Zygmund, {\em Trigonometric series. Vol. I, II. Third edition}. With a foreword by Robert A. Fefferman. Cambridge Mathematical Library. Cambridge University Press, 2002.



\end{thebibliography}
\end{document}